\documentclass[10pt,a4paper]{amsart}
\usepackage{amssymb}
\usepackage{amsmath,amsfonts,amsthm}
\usepackage{amstext,latexsym,bm}
\usepackage{graphicx}

\usepackage{xcolor}
 
\usepackage{amscd}
 
\usepackage{amssymb}
\usepackage{amstext}
\usepackage{mathtools}
\usepackage{bbm}

\newtheorem{theorem}{Theorem}
\newtheorem{corollary}[theorem]{Corollary}
\newtheorem{lemma}[theorem]{Lemma}
\newtheorem{proposition}[theorem]{Proposition}
\theoremstyle{definition}

\newtheorem{remark}[theorem]{Remark}
\newtheorem{definition}[theorem]{Definition}
\newtheorem{example}[theorem]{Example}

\renewcommand{\P}{\mathbb{P}}
\renewcommand{\emptyset}{\varnothing}
\renewcommand{\d}{\;\mathrm{d}}
\renewcommand{\theta}{\vartheta}
\newcommand{\C}{\mathbb{C}}
\newcommand{\R}{\mathbb{R}}
\newcommand{\N}{\mathbb{N}}
\newcommand{\M}{\mathcal{M}}

\newcommand{\E}{\mathbb{E}}
\newcommand{\e}{\mathrm{e}}
\newcommand{\F}{\vec{F}}
\newcommand{\X}{\mathbb{X}}
\newcommand{\1}{\mathbbm{1}}
\renewcommand{\hat}{\widehat}
\renewcommand{\tilde}{\widetilde}
\DeclareMathOperator*{\slim}{\star-lim}
\DeclareMathOperator{\Cov}{Cov}
\DeclareMathOperator*{\argmax}{arg\;max}
\numberwithin{equation}{section}
\numberwithin{theorem}{section}

\begin{document}

\title[Parametrized Families of Gibbs Measures -- Statistic Inference]{Parametrized Families of Gibbs Measures and their Statistical Inference}
\author{M. Denker} \email{manfred.denker@mathematik.uni-goettingen.de}
\address{University of G\"ottingen
Institut f\"ur Mathematische Stochastik, Goldschmidtstr. 7, 37077 G\"ottingen, Germany}
\author{M. Kesseb\"ohmer}\email{mhk@uni-bremen.de}
\address{University of Bremen, Institute of Dynamical Systems, FB 3 -- Mathematics and Computer Science, Bibliothekstr. 5
28359 Bremen, Germany}
\author{A. O.  Lopes} \email{arturoscar.lopes@gmail.com}
\author{S. R. C.  Lopes} \email{silviarc.lopes@gmail.com}
\address{Universidade Federal do Rio Grande do Sul Instituto de Matem\'{a}tica e Estat\'{\i}stica Av. Bento Gon\c{c}alves, 9500 91509-900 Porto Alegre, RS-Brazil}

\thanks{MK was supported by the Collaborative Research Center TRR 181 {\em ``Energy Transfers in Atmosphere and Ocean''} by the  German Research Foundation (DFG)}
\subjclass[2010]{62F02, 62F12, 62E20, 37A50, 37A10}
\keywords{Gibbs measure,  dynamical time series, maximum likelihood estimator, parametric hypothesis testing}
\date{08-02-2024}

\begin{abstract}  For  H\"older continuous functions $f_i$, $i=0,\ldots ,d$, on a subshift of finite type and  $\Theta\subset \mathbb \R^d$ we consider a parametrized family of potentials $\{F_\theta= f_0+\sum_{i=1}^d \theta_i f_i : \theta\in \Theta\}$. We show that the maximum likelihood estimator of $\theta$ for a family of Gibbs measures with potentials $F_\theta$  is consistent and determine its asymptotic distribution under the associated shift-invariant distribution. A second part discusses applications; from confidence intervals through testing problems to connections to Bernoulli distributions and stationary Markov chains.  
\end{abstract}
\maketitle

\section{Introduction}\label{sec:1}

Statistical inference for invariant distributions of stationary random sequences is traditionally investigated in  time series models or stationary Markov chains (see e.\,g.\@ \cite{Bro}, \cite{MMP}). The study of stationary sequences when such a modelling is  not possible
goes back to the 1970es when several authors studied nonparametric statistical inference for weakly dependent processes.  Main results in this direction are centered around extensions of asymptotic distribution results for classical nonparametric statistics in the independent case. To our knowledge the first  attempt to apply these results to genuine dynamics seems to be \cite{D1}.  First rigorous results --for dynamical systems as understood here-- are found in \cite{DK1} and \cite{DK2}, followed by \cite{Chu}. Since then the topic gained more interest, but was mostly confined to time series and their  analysis and prediction and to  statistical inference for parameters in stochastic differential equations, as a recent example we mention \cite{N} on likelihood analysis. 
Chang and Tong in \cite{ChanTong} quote: {\em ``H. Poincar\'e   recognized  that initial-value sensitivity is a fundamental source of randomness''}. Certainly, this is a point of view one can agree with. However, there are arguments as well in science that this phenomenon  should more accurately be called non-predictable. The difference seems to be merely philosophically and of no importance  in the axiomatic Kolmogorov model for probability. On the other hand, observing the recent discussion of EPR experiments in physics (\cite{HP}), there may be much more behind such an artificial distinction. The book \cite{ChanTong} provides a first attempt to describe statistical phenomena from a dynamic viewpoint. 

More recently, we mention, among others, the work in \cite{MMP} to transfer the classical parametric statistical analysis to the modern classical dynamical setup (for example, as presented in \cite{KH} and in this paper).  

Gibbs distributions play a fundamental role in statistical physics, dynamical systems and  ergodic theory. They build the core of the equilibrium theory and are studied as a separate topic, called thermodynamic formalism. Its fundamental mathematical results  go back to Ruelle and Bowen around 1970, though earlier fragments of the theory can be traced back up to the 1920.
They are defined as stationary probabilities which arise as a limit measure of push-forward distributions by H\"older continuous potentials. 

The typical examples of such dynamics are subshifts of finite type (cf. \cite{DGS}) and its factors like Anosov diffeomorphisms, rational functions and Markov maps on the unit interval, so name the most prominent applications. Most of these examples can be described as open and expanding maps (Ruelle expanding as they are termed in \cite{D3}). The probabilistic properties  of Gibbs distributions have been extensively studied since 1975 (\cite{Bo}) where it is shown that partial sums of H\"older-continuous functions satisfy a central limit theorem. Since stationary Gibbs measures are 
continued fraction mixing probabilities (\cite{ADU}), these partial sums are well approximable functions of $\psi$-mixing processes (\cite{Br}), thus many other probabilistic properties are immediate (cf. \cite{PS}, \cite{D2},\cite{DK01}, \cite{L3}). 

Much fewer problems have been considered in connection with statistics. While asymptotic distributional results have immediate  applications to nonparameteric statistics to estimate expectations (this goes back to the 70es) and multiple integrals
(\cite{DG} for a latest example), much less is known in the parametric and semiparametric cases.  The paper by Ji (\cite{Chu}) considers the family of Gibbs measures parametrized by their H\"older-continuous potentials $f$. Note that two Gibbs measures $\mu_f$ and $\mu_g$ are equal if and only if $f$ and $g$ are cohomologous. Thus the parameter space is infinite dimensional and can be identified with the space of H\"older-continuous functions modulo this equivalence relation of cohomology. The purpose in Ji's work is to estimate functionals of the Gibbs measures. It is shown that the maximum likelihood estimator based on the initial string of  the coordinate process estimates the true distribution and thus also the functional. It is, moreover, asymptotically efficient (like in the Cram\'er--Rao bound). This work is not in the main stream of classical parametric inference since the parameter space is not finite dimensional. 
 It should be noted that the same approach as in \cite{Chu} works in the present model of restricted classes of Gibbs measures. In this case one  should work with the complete knowledge of the   associated invariant distributions, instead of eigenmeasures as we do here.
 
A different approach to  treat parametric families of Gibbs measures can be found in \cite{MMN}, \cite{LoMe} and \cite{LV}. Their approach is a Bayesian one, starting with a parametrized family of  potentials, where the parameter space is compact and metric. They estimate $\theta$ through the exponential growth rate of the partition function. No maximum likelihood estimator
appears here, and in \cite{MMP} such estimators are only mentioned in abstract form.   Another interesting approach to statistical inference appeared in \cite{AK} where permutation statistical methods and entropy maximizing arguments are used to estimate invariant distributions. 

 This note studies the maximum likelihood estimator in the framework of parametric statistics.
This theory begins with the famous Neyman--Pearson lemma  and can be immediately applied to two Gibbs measures $\mu_0$ and $\mu_1$. It has been hardly observed in the literature for stationary measures that statistical analysis depends on the time series which has been observed and the  statistical problem which can be phrases with that knowledge. In order to illustrate that point, assume that a finite time series $\X_{n}\coloneqq (X_0,\ldots,X_{n-1})$ is observed which are the first $n$ steps of an infinite typical time series of the underlying Gibbs distribution.  
Therefore $\mu_0$ against $\mu_1$ cannot be tested based on the observed finite time series, one needs to redefine the test problem by testing the conditional distribution of $\mu_0$ against the conditional distribution of $\mu_1$, the condition is given by the $\sigma$-field generated by $\X_{n}$. Then the Neyman--Pearson Lemma is directly applicable and best tests can be explicitly constructed. Details are contained in \cite{DLL}.

Parametric maximum likelihood method needs a   family of Gibbs measures para\-me\-trized by a finite dimensional set. Due to its applications in dynamics  as driven by their potential  it is therefore natural to parametrize the underlying potential. Hence we take the family of Gibbs measures $\mu_\theta$ for $\theta\in \R ^{d}$ which have potentials
\[ F_\theta\coloneqq f_0+ f_{\theta} \;\;\mbox{ with }\;\; f_{\theta}\coloneqq \sum_{i=1}^d \theta_i f_i,
\]
where we slightly misuse the notation in the definition of $f_{\theta}$. 
If all the real-valued H\"older continuous functions $f_i$ are known, we have a parametric statistical problem, if some $f_i$ are not known the problem becomes semiparameteric. This paper concerns the parametric setup.

In the following, we will use  $\nu_{\theta}$ to denote the Gibbs measure, which is the unique normalized fixed point of the Perron-Frobenius operator with respect to the potential $F_{\theta}$, and  we will use $\mu_{\theta}$   to denote the corresponding unique shift-invariant Gibbs measure (the details of this will be provided at a later point). Like in the case of the Neyman-Pearson test problem 
we need to restrict the statistical problem to the situation when only    $\X_{n}$ 
is observable, thus redefining the statistical problems to that of the conditional distributions. 

 Statistical inference is concerned with two problems: Based on the observation $\X_{n}$ 
\begin{enumerate}
\item estimate the true parameter $\theta\in \Theta$  and
\item test the hypothesis $H_0\cong \theta \in \Theta_0\subset \Theta$ against $H_1\cong \theta\not\in \Theta_0$.
\end{enumerate}
In this note we address both questions keeping in mind that all statistics stay calculable from the initial data (here the potential). 
As a statistics we shall use maximizing principles, either a maximum likelihood of densities or  maximizing the basic thermodynamic relation for equilibria.
For the calculation of the maximum likelihood estimator $\hat \theta^n$  one needs to calculate the values $\nu_\theta(C)$ for any $n$-cylinder $C$. This can be done from the knowledge of the basic potentials $f_i$ by  
\[\nu_\theta(C)=\lim_{m\to\infty} \int \tilde{\mathcal  L}^{m}_\theta(\1_C)\d\mu \Huge{ / }\int \tilde{\mathcal L}^{m}_\theta(\1) \d\mu\]
for any initial probability $\mu$, where $\tilde{\mathcal  L}_{\theta}$ denotes the {\em transfer operator}
\[ \tilde {\mathcal L}_{\theta}f(x)= \sum_{\sigma(y)=x} f(y)\e^{F_{\theta}(y)},\]
$\sigma$ the shift transformation and $\1_{A}$ the indicator function on the subset $A$ and $\1$ the constant $1$-function.  It is known that this converges exponentially fast (e.g.\ \cite{BS}, Theorem 1 for a practicable algorithm in case of shift-invariant Gibbs measures; the non-invariant and non-normalized case is similar, since its eigenmeasure is absolutely continuous with respect to the shift-invariant Gibbs measure with bounded logarithm of its derivative).

Alternatively, instead of maximizing the probability densities one can use the estimator $\tilde \theta$ given by
\[ \max_{t} nP(F_t)- S_nF_t(x)= nP(F_{\tilde \theta})- S_nF_{\tilde \theta}(x)\]
This is well defined if the covariance matrix is positive definite. In this case one needs to calculate the maximal eigenvalue of the transfer operator which is much harder because the operator acts on an infinite-dimensional Banach space.

We shall show that both estimators will be asymptotically equivalent.

First  of all, one would like to use the maximum likelihood  estimator of $\theta$ based on  observations $X_0,\ldots,X_{n-1}$ which have a finite state space $E$ for shift-invariant Gibbs distributions.  For given $(c_i)\in E^{\{0,\ldots, n-1\}}$ this estimator is determined by maximizing
\[ \mu_\theta(X_0=c_0,\ldots,X_{n-1}=c_{n-1}), \qquad \] 
 the densities of the conditional distributions with respect to the dominating unit mass distribution on all $n$-cylinders. 
The first problem we are facing here is to write the densities explicitly as functions of $\theta$ through their potential functions. This is not possible for invariant Gibbs measures where - in general - the potential is not explicitly calculable, only approximatively. In general, the invariant Gibbs measures and their densities are numerically  calculable, but it is hard to find a suitable form for its derivative and to deduce the distribution of the corresponding maximum likelihood estimator. The case of independent, finite state processes   is an exception.
Instead of using the invariant Gibbs measure, 
 there is the equivalent probability $\nu_\theta$ with bounded Radon-Nikodym derivative which allows an explicit form of its density. Therefore we study the maximum likelihood estimator for the family $\nu_\theta$. Here, it turns out that one cannot solve the equations 
\[\frac {\partial}{\partial\theta_i}\nu_\theta([X_0,\ldots,X_{n-1}])=0\;\;\;\;(1\le i\le d)\]  
explicitly (like in the classical cases when deriving maximum likelihood estimators). Instead, we are using Taylor expansion for the first order derivatives and show the existence of a maximum using thermodynamics and thus show that the first order derivatives at the maximum vanish.

One more comment needs to be made here.  The family of Gibbs measures considered here depend on the parameter $\theta$ and a finite number of basic potentials $f_i$. It turns out that  (in many cases) there is a bijection between the parameter space and the Gibbs family through their integrals $\int f_i \d\mu_\theta$. These integrals can easily be estimated like in \cite{Chu}, but this does not allow statistical inference on the parameter $\theta$ unless one knows the bijection. In Section 6, Example \ref{ex:6.4}, we show that dependence using a simple example. It then becomes clear  that the plug-in  estimator derived from that bijection does not simplify the problem essentially.

It turns out that under canonical assumptions the estimator is consistent and has probabilistic properties governed by the distribution $\mu_{\theta}$. In particular, we determine the asymptotic distribution of $\hat \theta^n$, the maximum likelihood estimator based on $\X_n$.  It then follows that confidence intervals for $\theta$ can be constructed, as well as test problems about the parameter $\theta$ can be solved.

We also show that the estimator is asymptotically efficient in some weaker sense: The variance of $\hat \theta^n$ decreases of the order $1/n$, the same order as given by the optimal rate for the best unbiased estimator given by the Cram\'er--Rao bound. A particular parametrized family is given by potential functions $f_0=0$ and $f_i=\1 _{[i]}$ for $i\in E$ and the parameter space 
\[\Theta\coloneqq\left\{\theta: 0\le \theta_i\le 1;\ \sum_{i\in E} \theta_i=1\right\}.\]
 In this case the measures $\nu_\theta$ and $\mu_\theta$ are equal, and $\mu_\theta$ is the Bernoulli measure on $E^\mathbb N$ with $\mu_\theta([i])=\theta_i$ for $i\in E$. In this case the maximum likelihood estimator for $\theta$ is a {\em uniformly unbiased, minimum variance estimator (UMV) estimator}. Since the parameter arising from the parametrization in thermodynamics is different, but a function $h$ of the above parametrization, we obtain a maximum likelihood estimator which differs by a random factor form the estimator obtained  by  plugging in the UMV estimator into $h$.

 The paper is organized and written to make it more comfortable to read for researchers in dynamical system theory, statistics and applied sciences. Thus it includes a somewhat lengthy discussion of background material in thermodynamics and statistics.  It discusses applications in great detail, but cannot be complete in that respect. It leaves many directions of further research. 
 
 The organization is as follows. Section \ref{sec:2} contains basic definitions and the main theorem. Note that the definition of the maximum likelihood estimator is slightly different from the classical one, because one needs to avoid that the estimator escapes the parameter space. We also connect the existence of the estimators to the concavity of the pressure function. Section \ref{sec:3} contains background material for the proof of the main theorem in Section \ref{sec:4}. In order to make the paper better accessible to non-specialists in thermodynamics we prove/reprove analytic properties of Gibbs measures in Section \ref{sec:3.1} (cf. \cite{Chu, Lal1, Lal2, PP}). Moreover, we recall non-standard results to deal with the moments and distributions of covariance matrices arising for stationary measures. Note that covariance $\frac 1n \int S_nfS_ng \d\mu$ converge to the covariance of the limiting normal distribution if both functions $f$ and $g$ are H\"older continuous and centered (under an invariant Gibbs measure) while $\frac 1n S_nfS_ng$ converges to a weighted $\chi^2$-distribution if the functions are centered (Theorem \ref{theo:3.10}, see \cite{DG}, a similar result appeared in \cite{LN}). This distribution appears in our calculation of the maximum likelihood estimator as some marginal distribution. In Section \ref{sec:5} we apply the asymptotic properties of the maximum likelihood estimator to construct maximum likelihood tests for simple hypotheses and tests for detecting some influence of a particular potential $f_k$. Finally, in Section \ref{sec:6} we provide several examples to illustrate the results and, in particular construct confidence intervals and compare the result with the classical estimation when the random sequence is independent or is a Markov chain. The final section discusses the equivalence of the two estimators, the maximum likelihood estimator and the maximum potential estimator. It also deals with an extension of the theory for $C^3$-parametrization (instead of a linear one).

\section{Main results}\label{sec:2}
 
  \subsection{Maximum likelihood estimator}\label{sec:2.1}
  
We begin describing the statistical model, which consists of a $d$-dimensional parametrized family of Gibbs measures 
$(\nu_\theta)_{ \theta\in \Theta}$, $\Theta\subset \mathbb R^d$, defined on a toplogically mixing subshift of finite type 
$\Omega\subset \{1,2,..,a\}^{\mathbb N_0}$, $a\ge 2$ and $\mathbb N_0=\{0,1,2,\ldots\}$, with shift transformation $\sigma:\Omega\to\Omega$ (see \cite{DGS} for a definition) and 
  observations $\mathbb X=(X_n)_{n\in \mathbb N_0}$ defined on $\Omega$ where $X_n$ denotes the projection onto the $n$-th coordinate, thus has the state space $\{1,\ldots,a\}$. The Gibbs measures $\nu_\theta$ are defined by potentials $F_\theta$ as given above, where $f_i$, $0\le i\le d$ are H\"older continuous functions on $\Omega$. The space of all shift-invariant probability measures on $\Omega$ will be denoted by  $\mathcal M(\Omega)$ and  $\mu_{\theta}\in\mathcal M(\Omega) $ denotes the associated normalized shift-invariant version of $\nu_\theta$ (see  Section \ref{sec:3.1}).  We set $[x]\coloneqq [x_0,\ldots,x_{n-1}]\coloneqq \{\omega \in \Omega: \omega _i=x_i, 0\le i<n\}$, which we call a {\em cylinder of length} $n$ (or $n$-cylinder for short) and the set of {\em admissible words of length} $n$ by $\Omega^{n} \coloneqq \{ x=(x_0,\ldots,x_{n-1})\in  \{1,2,..,a\}^n : [x]\neq \emptyset \}$. For  $\theta\in \Theta$, we can   define the marginal distributions of $(X_k)_{k\in \mathbb N_0}$  via
  \[ 
\P(X_0=x_0,\ldots,X_n=x_n)= \nu_\theta([x_0,\ldots,x_n]),\;\;n\in \mathbb{N}_{0}.
\]
 (We use here $\P$ to denote an `unspecified' probability in accordance to common practice in stochastics.) If this holds for some $\theta$ we say that $\mathbb X$ is a random sequence drawn from the distribution $\nu_\theta$. For a function $f:\Omega\to \C $ we let $S_{n}f\coloneqq \sum _{k=0}^{n-1}f\circ \sigma^{k}$, and for vector-valued $G\coloneqq (g_{1},\ldots,g_{d}):\Omega\to \C^{d} $ this defintion caries over to $S_{n}G\coloneqq (S_{n}g_{1},\ldots,S_{n}g_{d})$ by addition in the vector space; similarly, the integral $\int G \d\mu\in \C ^{d}$ is defined component-wise.  With $\vert \,\cdot\,\vert_{p}$ we denote the $p$-norm in $\C^{d}$, $p\in [1,\infty]$ and particularly, for the euclidean norm $\vert \,\cdot\,\vert\coloneqq\vert \,\cdot\,\vert_{2}$. 
 For a cylinder $[x]$ of length $n$ and $G$ an $\R^{d}$-valued function, we set 
 \[S_{[x]}G\coloneqq \bigg(\sup_{\omega \in [x]}S_{n}g_{1}(\omega),\ldots,\sup_{\omega \in [x]}S_{n}g_{d}(\omega)\bigg).\]

\begin{definition}\label{def:2.1} Let $\{\nu_\theta:\theta \in \Theta\}$, $\Theta\subset  \R^d$, be a family of Gibbs measures with potentials $ F_\theta$,
where the defining function is given by the vector  $\F\coloneqq(f_{0},\ldots,f_{d})$ of    H\"older continuous and real-valued functions $f_i$, $0\le i\le d$. 
For the control parameter $\eta>0$, we call any estimator $\widehat \theta_n:\Omega^{n} \to \Theta$,  an {\em $\eta$-maximum likelihood estimator  for $\theta$}   
if for $x=(x_{0},\ldots ,x_{n-1})\in \Omega ^{n}$ we have
\[
 \nu_{\widehat \theta_n(x)}([x])= \max_{s\in\Theta_{n}^{\eta}(x) } \nu_s([x]),
 \]
where 
\begin{equation}\label{eq:2.1}
 \Theta_{n}^{\eta}(x)\coloneqq \left\{ s\in \Theta\cap\left[- 1/ \eta,  1/\eta\right]^{d}:  \left| \frac1n S_{[x]} \F- \int \F\d\mu_{s}\right|_{\infty}\le \eta^2\right\}.
 \end{equation}
 \end{definition}

\begin{definition}\label{def:2.2} Let $\Theta\subset \mathbb R^d$ be an open  subset and $t\mapsto m_t$ be a continuous map defined on $\overline \Theta$ with values in the space of probability measures of a compact metric space. A sequence $t_n\in \overline \Theta$ is said to $\star${\em -converge to} $t\in \overline \Theta$ if $m_{t_n}$ converges weakly to $m_t$ and in this case we write $\slim t_{n}=t$.
\end{definition} 
\begin{remark}\label{rem:2.3}  Let $\Theta$ be open and bounded.
If $(t_n)\in \overline \Theta^{\N}$ is $\star$-convergent to $t \in \overline\Theta$ and if $x\to m_x$ is injective on $\overline \Theta$, then $t_n\to t$ also in the euclidean topology. To see this let $s$ be an $\star$-accumulation point of $(t_n)$. By continuity, $m_s$ is an accumulation point of $(m_{t_n})$,  hence $m_s=m_t$ and so by injectivity of the map  $x\mapsto m_x$ we get $s=t$. 
\end{remark}

In what follows, for a vector $v=(v_1,\ldots,v_d)\in \mathbb R^d$, we freely switch to matrix notation by writing $v^t $ for the vector in column form, where $v$ is understood in row form.
The inner product of two vectors $u$ and $v$ is denoted by
 $\langle u,v\rangle$.
 
 In this note we prove the following main result for the existence of the maximum likelihood estimator.
 \begin{theorem}\label{theo:2.4} Let $(X_k)_{k\in \mathbb N_0}$ be a random sequence drawn from Gibbs measures $(\nu_\theta)_{\theta\in \Theta}$ with H\"older continuous potential $F_\theta=f_{0}+f_{\theta}$, where $\emptyset \ne \Theta\subset \mathbb R^d$ is open. Then there exists a sequence $(\eta_n)_{n\ge 1}$ of positive numbers converging to zero such that:

For sufficiently large $n\in \mathbb N$ there exists an $\eta_n$-maximum likelihood estimator $\widehat \theta_n=(\theta_{n,1},\ldots,\theta_{n,d})$ based on $\X_{n}=(X_0,\ldots,X_{n-1})$. This estimator is strongly consistent in the sense that if $(X_k)_{k\in \mathbb N_0}$ is  drawn from the distribution $\nu_\theta$,  then  $\widehat \theta_n$ $\star$-converges  a.s.\@ to $\theta$.
\end{theorem}

\subsection{Maximum potential estimator for invariant Gibbs measure}\label{sec:2.2}

In this section we discuss the existence of a consistent estimator from a theoretical viewpoint (the reader may consult with \cite{PU} or \cite{CM24} for more insight). This estimator will be called maximum potential estimator (MPE) since it arises from maximizing functions involving the pressure of the underlying potential $F_{t}$, 
\[ P(F_t)=\sup\left\{ h_\mu(T)+\int F_t \d\mu:\mu\in \mathcal{M}(\Omega)\right\}.\] 
 We find conditions for the existence of these maxima. A different question is whether such estimators are computable, and another arises in their asymptotic distribution to achieve statistical relevance. The latter two problems are dealt with in the following sections, and we will see that the maximum likelihood estimator has better properties in view of its computability. For the existence of the MPE, we consider two cases separately.

To begin with, let $H$ denote the subspace of all H\"older continuous functions $g$ on $\Omega$. We define an equivalence relation on $H$ via $f\sim g$ for $f,g\in H$ if $f$ and $g$ are  {\em cohomologous}, that is there exists a continuous function $h$ and a constant $c\in \R$ such that $g-f=c+ h-h\circ \sigma$  where $\sigma$ denotes the shift transformation on $\Omega$ as before.  

\begin{definition}\label{def:2.5}
We say that a family  $F\coloneqq (f_1,\ldots ,f_d)$   of  H\"older continuous functions on $\Omega$ is {\em linearly independent as cohomology classes} if for all $s\in \R^d$ the relation $f_{s}=\sum_{i=1}^{d}s_{i} f_i\sim 0$ implies $s=0$.
\end{definition}

For  $\mu\in \mathcal{M}(\Omega)$   and a $\mu$-integrable  function $f$  we write for its centralized version  $\tilde f^{\mu}\coloneqq f-\mu(f)$, $\mu(f)\coloneqq \int f \d \mu$. If $\mu=\mu_{\theta} $ is an invariant Gibbs measure for the potential $F_{\theta}$, then  we also use the short hand notation $\tilde f^{\theta}\coloneqq  \tilde f^{\mu_{\theta}}$. Recall that the {\em asymptotic covariance} of two square integrable functions $f,g$ with respect to $\mu$ is defined by
\[  \Cov _{\mu}(f,g)\coloneqq \lim_{m\to\infty} \frac 1m \int S_m \tilde f^{\mu}  \, S_m  \tilde g^{\mu} \d\mu\] and   we define the {\em covariance matrix} of $F=(f_1,\ldots,f_d)$ by  
$$\Sigma^{\mu} =(\Sigma_{ij}^\mu)_{1\le i,j\le d}\coloneqq  \Cov _{\mu}(F)\coloneqq (\Cov _{\mu}(f_{i},f_{j}))_{1\le i,j\le d}.$$

We believe that the following theorem is standard (cf.\ \cite{CM24}) but we give the proof for completeness.
\begin{theorem}\label{theo:2.6} Let $F\coloneqq(f_1,\ldots ,f_d)$ be a vector of H\"older continuous observables and let $\mu$ denote some invariant Gibbs measure. 
Then  $F$ is linearly independent as cohomology classes if and only if the associated covariance matrix $\Sigma^{\mu}\coloneqq  \Cov _{\mu}(F)$  is positive definite. 
 \end{theorem}
\begin{proof} First note that by the mixing property of Gibbs measures on mixing subshifts of finite type the asymptotic covariances $\Cov _{\mu}(f_{i},f_{j})$ exist for all $0\le i,j\le d$ and are linear in both arguments and symmetric.
 The property of the familiy $(f_1,\ldots ,f_d)$ to be linearly independent as cohomology classes is equivalent to the fact that   for   each $s\in \R^{d}\setminus\{0\}$, the function $f_{s}   $ is not cohomologous to zero which holds  if and only if, for all $s\in \R^{d}\setminus\{0\}$
\[0<\Cov _{\mu}(f_{s},f_{s})=s \Sigma^{\mu}s^{t},\] 
  as is well known (cf.\@    Lemma \ref{lem:3.5}). This gives the claim.
\end{proof} 

\paragraph{Full rank.}

Let $\Theta\subset \R^{d}$ be an open subset and let $\mu_{t}\coloneqq \mu_{F_{t}}$ denote the invariant Gibbs measure for the potential $F_{t}$, $t\in \Theta$. Fix $\theta_0\in \Theta $,  and  random samples $\X_{n}\coloneqq (X_{0},\ldots, X_{n-1})$ drawn with respect to $\mu_{\theta_0}^{n}$($n\in \N$), the marginal distribution on $n$-cylinders of $\mu_{\theta_0}$. We assume here that the covariance matrix $\Sigma^{\mu_{\theta_0}}$ has full rank. 
 By the Gibbs property we have
\begin{equation}\label{eq:2.2}
-K\leq \log\mu_{\theta_0}[\X_{n}]-S_{\X_{n}}F_{\theta_0}+nP(F_{\theta_0})\leq K.
\end{equation}
 This means that $ \log\mu_{\theta_0}[\X_{n}]$ and $n(\langle \alpha_{n}, (1,\theta_0)\rangle -P(F_{\theta_0}))$ with $$\alpha_{n}=(\alpha_n^i)_{0\le i\le d}\coloneqq S_{\X_{n}}\F/n,$$  are $2K$ close for all $n\in \N$ and independent of $\X_{n}$, where $\F=(f_0,\ldots,f_d)$. By ergodicity we have $\lim _{n}\alpha_{n}= \int \F\d\mu_{\theta_0}$ a.\,s.  
Our cohomology assumption on $f$ implies that $t\mapsto  \int F_{t}\d\mu_{\theta_0}-P(F_t)$ is strictly concave and attains a global maximum at $t=\theta_0$, for by Lemma \ref{lem:3.2}
the first partial derivatives vanish and the second derivative is given by  $\Sigma^{\mu_{\theta_0}}$. Since $\Sigma^{\mu_{\theta_0}}$ has full rank there is $\theta\in \Theta$ such that
\[ \Sigma^{\mu_{\theta_0}}( \theta-\theta_0)=  (\alpha_n^i-\int f_i \d\mu_{\theta_0})_{1\le i\le d}.\] 

Likewise $\langle\alpha_{n},(1,t)\rangle-P(F_{t})$ is strictly concave as well and has a unique maximum in some $\tilde\theta_{n}\in \R^{d}$ with $\nabla P(F_{s})(\tilde\theta_{n})=\alpha_{n}$ and we have 
\[\lim \tilde\theta_{n}\to \theta_0,\] 
where $\theta_0$ is uniquely determined by  $\nabla P(f_{s})(\theta_0)=\left(\int f_{i}\d\mu_{\theta_0}\right)_{i=1,\ldots, d}$. Let 
 us call  $\tilde\theta_n$ the {\em maximum-potential estimator} (MPE).

 Taylor expansion of  $\langle\alpha_{n},t\rangle-P(f_{t})$  in $\tilde\theta_{n}$ guarantees that 
\[
\vert\langle\alpha_{n},\tilde\theta_{n}+s\rangle-P(f_{\tilde\theta_{n}+s})-\langle\alpha_{n},\tilde\theta_{n}\rangle+P(f_{\tilde\theta_{n}})\vert\geq C \lVert s\rVert^{2}.
\]
where $C>0$ is the minimal eigenvalue of $\Sigma^{\mu_{\tilde\theta_{n}}}$. 
These observations imply that the maximum of $\log\mu_{t}[\X_{n}]$,   obtained in some $\hat\theta _{n}$ (maybe not unique),   lies in a  $\sqrt{K/n}/C$-neighborhood of $t_{n}$.
Combining the above gives $\hat\theta _{n}\to \theta_0$, hence consistency of the MLE as well.

We summarize in 
\begin{theorem}\label{theo:2.7}
  Let $(F_t)_{t\in \Theta}$  be a family of H\"older potentials with open parameter space $\Theta\subset \mathbb R^d$ defined as before via $\F \coloneqq (f_{0},\ldots,f_{d})$.  If the covariance matrix $\Cov_{\mu_{\theta_{0}}}(F)$ in $\theta_0\in \Theta$ has full rank, then there exists a neighbourhood  $U$ of $\theta_0$ such that with $\alpha_{n}\coloneqq n^{-1}S_{\X_{n}}\F$, the  estimator  defined by
  \[ \tilde \theta_n \equiv \argmax_t \{\langle\alpha_{n},(1,t)\rangle-P(F_{t})\}\] 
  exists and converges to $\theta_0$ a.s. 
  \end{theorem}
 
 This argument works for both the invariant and the non-invariant Gibbs measures.
 Other properties of the MPE are deferred to Section \ref{sec:7.2}.
 
\paragraph{General case.}

Now we assume that $\mathrm{Ker}(\Sigma^{\mu_{\theta}})=\{v_{1},\ldots ,v_{\ell}\}\eqqcolon V$ has dimension $ \ell<d$. 
Then we observe
\[
\forall v\in V:\mu_{t+v}=\mu_{t}.
\]
This means that $\log \mu_{t}$ is constant along the subspace $V$  and we can  restrict our attention to $ V^{\perp}\simeq\R^{d}/V\simeq\R^{d-\ell }$.  Using \eqref{eq:2.2}  and the fact that $t\mapsto P(F_{t})$  restricted to $V^{\perp}$ is strictly convex, we obtain a unique  point $t_{n}\in V^{\perp}$ maximizing $\langle \alpha_{n},(1,t)\rangle+P(F_{t})$ in $V^{\perp}$ as well as a maximum  of $t\mapsto \log\mu_{t}[\X_{n}]$ in $s_{n }\in  V^{\perp}$, which lies   in a $\sqrt{K/n}/C$-neighborhood of $t_{n}$; here $C>0$ denotes the minimal non-zero eigenvalue of $\Sigma^{\mu_{\theta}}$. Let $p^{\perp}$ denote the orthogonal projection of $\R^{n}$ onto $V^{\perp}$.  Then any element in $\Theta\cap (p^{\perp})^{-1}(\{s_{n}\})$ is a possible MPE.
 Hence, we can choose an estimator $\hat\theta_{n}$ as an element with minimal euclidian norm in 
\begin{equation*}\Theta_{n}\coloneqq\{t\in \overline \Theta:\mu_{t}([\X_{n}])=\max_{s\in \overline \Theta}\mu_{s}([\X_{n}])\}
\end{equation*} that is,
\[
\hat\theta_{n}\in \Theta_{n}\qquad \mbox{ and }\qquad \|\hat\theta_{n}\|= \inf_{t\in \Theta_{n}}   \|t\|.
\]
  
\subsection{Asymptotic distribution of the MLE}\label{sec:2.3}

 By Theorem \ref{theo:2.6}  and its consequences  maximum likelihood estimators exist if the H\"older continuous potentials $f_1,...,f_d$ are linearly independent as cohomoloy classes.In this case the map $t\mapsto \nu_t$ is injective so Theorem \ref{theo:2.4} can be used to calculating an approximate version of this estimator. In this case we can speak of a uniquely defined MLE. We  thus need to investigate its distributional properties.

\begin{theorem}\label{theo:2.8} Let $f_0,\ldots,f_d$ be H\"older continuous functions on $\Omega$ and $\theta\in \Theta$ where  $\emptyset \ne \Theta\subset \mathbb R^d$ is open and suppose that  the vector $F\coloneqq (f_1,\ldots,f_d)$ is linearly independent   as  cohomology classes. Assume that  $(X_k)_{k\in \mathbb N_0}$ is a random sequence drawn from the Gibbs measure $\nu_\theta$ (resp. $(\mu_\theta)$)  with  respect to the  potentials
\[ F_\theta=f_0+\sum_{i=1} \theta_if_i \quad f_i:\Omega\to \mathbb R,\ 0\le i\le d.\]  
Then 
\begin{enumerate}
\item  The maximum likelihood estimator  $\widehat \theta_n$ is {\em consistent}, more precisely, this means that  $\lim_{n\to\infty} \widehat \theta_n (X_0,\ldots,X_{n-1}) = \theta$ a.s..
\item The asymptotic covariance matrix   $\Sigma^{\mu_{\theta}}=\Cov_{\mu _{\theta}}(F)$ is invertible. We have that the matrix
\[ G(m_1,\ldots,m_d)= (m_1,\ldots,m_d)^t (m_1,\ldots,m_d)-\Sigma^{\mu_{\theta}}\]
is invertible for almost all $(m_1,\ldots,m_d)\in \mathbb R^d$ with respect to the $d$-dimensional normal distribution $\mathcal{N}(0,\Sigma^{\mu_{\theta}})$.
\item  $\sqrt{n}(\widehat \theta_n-\theta)$  converges in distribution  to the random variable  $G^{-1}(N)N^{t}$, where where $N\sim \mathcal{N}(0,\Sigma^{\mu_{\theta}})$ and $G^{-1}(N)$ denotes the inverse of the matrix $G(N)$.
\item $\widehat \theta_n$ is {\em asymptotically efficient} in the sense that 
\[ \int \left |\widehat \theta_{n}-\theta\right |_2^2 \!\!\d\mu_\theta\le \frac 1n  \int \left\| G^{-1}\left((n^{-1/2} S_n\widetilde f^{\theta}_i)_{i=1}^{d}\right)\right\|^2 \sum_{j=1}^{d} \frac 1n \left(S_n\widetilde f_j^\theta\right)^2\d\mu_\theta,\] 
where the norm on the right hand side denotes the operator norm with respect to the euclidean norm and  is allowed to  be infinite.

Moreover, 
\[\lim_{n\to\infty} n \int (\widehat \theta_n-\theta)^t (\widehat \theta_n-\theta) \d\mu_\theta =   \E N(G^{-1}(N))^2 N^{t}.\]
\end{enumerate} 
\end{theorem}

The proof of this result is contained in Section \ref{sec:4}. The next section recalls definitions, the setup and preliminaries.

\begin{remark}\label{rem:2.9}
 Theorem \ref{theo:2.8} 4. is connected to the Cram\'er--Rao bound and  the Fisher information in as much that the asymptotic rate of decrease is $1/n$. The latter is a matrix defined via  the density 
for $\nu_\theta$ and it  satisfies  
\begin{equation}\label{eq:2.3} 
\lim_{n\to\infty} \frac 1n I^n(\theta) =\lim_{n\to\infty} \frac 1n \int S_n\widetilde f_i^\theta S_n\widetilde f^\theta_j \d\nu_\theta.
\end{equation}
This becomes clear from the proof of Theorem \ref{theo:2.8} and Lemma \ref{lem:4.5} below. 
It should be noted that the limit in (\ref{eq:2.3}) exists since the coordinate process under the probability measure $\nu_\theta$
is non-stationary, but $\psi$-mixing; hence the standard proof for the stationary case works here as well. 
Likewise the same type of limit exists for the stationary distribution $\mu_\theta$ and we obtain in this situation $\Sigma^{\mu_{\theta}}$. 
This expression is called in \cite{Chu} the {\em asymptotic Fisher information}, see (\ref{eq:3.11}, Proposition \ref{prop:3.9}) below in Subsection \ref{sec:3.4}. 
 \end{remark}

\section{Preliminaries}\label{sec:3} 

This section is used to formulate the problem in more detail, to recall statistical terminology and facts from thermodynamic formalism. 

\subsection{Gibbs measures}\label{sec:3.1}

We begin describing the   thermodynamic formalism.  
For $a\in\mathbb N$ let $\Omega=\{(x_n)_{n\ge 0}: a_{x_nx_{n+1}}=1\}\subset\{1,\ldots ,a\}^\mathbb N$ be defined by a $\{0,1\}$-valued incidence matrix $(a_{ij})_{1\le i,j\le a}$. This set is called a subshift of finite type over $a$ letters. It carries the natural product topology and thus also the natural Borel $\sigma$-field.
Denote by
$\M(\Omega)$  the set of $\sigma$-invariant  probabilities, where $\sigma :\Omega\to\Omega$ denotes the left shift map. We fix   $  (f_{0},\ldots,f_{d})$,  $d\in \mathbb N$, a vector  of  H\"older potentials defined on $\Omega$.

It is well known (cf. \cite{Bo}) that
for each $t \in  \mathbb{R}^d$ there exists a unique Gibbs  probability $\nu_t$  for the potential $F_t\coloneqq  f_0+ \sum_{i=1}^d t_i f_i$ which  is completely characterized by Bowen's formula (\cite{Bo}), i.\,e.\@  $\mathcal{L}^{*}_{F_{t}}\nu_{F_{t}}=\e^{P(F_{t})}\nu_{F_{t}}$ with $\mathcal{L}^{*}_{F_{t}}$ denoting the dual of the Perron-Frobenius operator defined below  and where $\e^{P(F_t)}$ is the maximal eigenvalue. For some $D(t)>0$ the Gibbs property holds:
\begin{equation} \label{eq:3.1}
D(t)^{-1} \le \frac {\nu_t([c_0,\ldots ,c_{n-1}]}{\exp(-nP(F_t)+ S_{n}F_tz)} \le D(t)
\end{equation} 
where  $n\ge 1$,
$[c_0,\ldots ,c_{n-1}]:=\{ \omega\in \Omega |\  \omega_i=c_i,\ \forall 0\le i<n\}\subset \Omega$ denotes a general cylinder set of length   $n$, $z$ is any point in $[c_0,\ldots ,c_{n-1}]$ and $S_nh(x)= \sum_{k=0}^{n-1} h(\sigma^k(x))$ for $x\in \Omega$ and $h:\Omega\to\mathbb R$ is measurable. Note that here $\nu_t$ is not necessarily shift invariant. The quantity
$P(F_t)$ is a constant and called the pressure  of the potential $F_t$. 
We will see later that $\log D(t)=\mathcal{O}(\|t\|)$ for large $\|t\|$.

Throughout the paper we make the convention to write $C$ to denote a cylinder set $[c_0,\ldots ,c_{n-1}]:=\{x\in \Omega: x_i=c_i\ \forall\ 0\le i <n\}$, sometimes without mentioning the letters $c_i\in\{1,\ldots ,a\}$.

Now we fix $n\ge 1$. The random variables $X_0,\ldots ,X_{n-1}$ are assumed to take values in $\{1,\ldots ,a\}$ and have the joint  distribution given by the probabilities
\[ \nu_t^n(X_0=c_0,\ldots ,X_{n-1}=c_{n-1}) =\nu_t([c_0,\ldots ,c_{n-1}])\] 
for all  $(c_0,\ldots ,c_{n-1})\in \Omega^{n}$. Because of its definition $X_0,\ldots ,X_{n-1}$ is a finite time series, which is not necessarily stationary because $\nu_t$ may  not be
shift invariant.

 Instead of using (\ref{eq:3.1}) we need the exact form of $\nu_t([c_0,\ldots ,c_{n-1}])$ which can be derived using a functional analytic representation.

For each $t\in \mathbb R^d$ the Perron-Frobenius (or transfer-) operator $\mathcal L_t$ for the potential $F_t$ is defined for bounded measurable functions $g:\Omega\to\mathbb R$ as
\begin{equation*}
 \mathcal L_tg(x) =\sum_{x=\sigma y}  g(y)\e^{F_t(y)}.
 \end{equation*}
Throughout the paper  we shall use the convention $kx=(k,x_0,x_1,\ldots )$ for $x=(x_0,x_1,\ldots )\in \Omega$ and $k\in \{1,\ldots,a\}$ as long $kx\in \Omega$. Likewise we shall use the convention $cx=(c_0,\ldots ,c_{n-1},x_0,\ldots )$ for   admissible  $n$-words   $(c_0,\ldots ,c_{n-1})\in\{1,\ldots,a\}^{n }$  and points $x\in \Omega$.
Classical results ensure that $\mathcal L_t$ acts on the space of continuous functions $C(\Omega)$ and on the space of H\"older-continous functions as well. The dual operator $\mathcal L_t^*:C(\Omega)^*\to C(\Omega)^*$ has a maximal eigenvalue $\Lambda_t$ and a unique normalized eigenmeasure $\nu_t$ for this eigenvalue. This eigenmeasure is called the  {\em Gibbs measure} or Gibbs distribution for the potential $F_{t}$. There is a H\"older-continuous eigenfunction $\varphi_t>0$ for the operator $\mathcal L_t: C(\Omega)\to C(\Omega)$ and  eigenvalue $\Lambda_t$. The invariant Gibbs measure for the potential $F_t$ is then defined as $\mu_t=\varphi_t \nu_t$ which we assume to be normalized, i.\,e.\@  $\int \varphi_t \d\nu_t=1$. It is also well known that $\Lambda_t=\e^{P(F_t)}$ where $P(F_t)$ denotes the pressure of the potential $F_t$.

The exact form of $\nu_t(C)$ for a cylinder $C\coloneqq [c_0,\ldots ,c_{n-1}]$   is then
\begin{equation}\label{eq:3.2} \nu_t(C)=\int \1 _C \d\nu_t= \Lambda_t^{-n} \int \mathcal L^n \1 _C \d\nu_t=
\int \e^{-nP(F_t)+S_nF_t(cx)}\d\nu_t(x).
\end{equation}

The measure $\mu_t$ also arises as an eigenmeasure for a transfer operator $\widetilde {\mathcal L_t}$ -- in the same way as  $\nu_t$ is obtained from $\mathcal L_t$ -- by replacing the potential  $F_t$ by $\tilde F_t= -P(F_t) +F_t +\log\varphi_t-\log [\varphi_t\circ \sigma]$. Thus it is also a Gibbs measure and   invariant, since it is known that the eigenfunctions $\varphi_t$ are H\"older continuous. In fact 
\begin{eqnarray*}
 \mu_t(C) &=& \int \1 _C \varphi_t \d\nu_t= \Lambda_t^{-1}\int \mathcal L_t [\1 _C \varphi_t] \d\nu_t= \int \mathcal L_t [\e^{-P(F_t)}\1 _C \varphi_t] \d\nu_t \\
 &=& \int \sum_{k=1}^a \1 _C(kx) \frac {\varphi_t(kx)}{\varphi_t (\sigma(kx))} \e^{F_t(kx)}\varphi_t(x) \d\nu_t(x)  \\
 &=& \int \sum_{k=1}^a \1 _C(kx) \e^{-P(F_t)+F_t(kx) +\log\varphi_t(kx)-\log\varphi_t(\sigma(kx))} \d\mu_t(x)\\
 &=& \widetilde{\mathcal L_t}^* \mu_t(C),
 \end{eqnarray*}
 Therefore by (\ref{eq:3.2}),
 \[ \mu_t(C)= \int \e^{-nP(F_t)+S_nF_t(cx)+\log\varphi_t(cx)-\log\varphi_t(\sigma(cx))} \d\mu_t(x).\] 
 This last equality shows that it is preferable to use the MLE for $\nu_t$ and not for $\mu_t$.

\subsection{Perturbation theory for Gibbs distributions}\label{sec:3.2}

Perturbation theory for transfer operators determines  its analytic properties. This has been studied by many authors (see e.g.\@ \cite{PP}, \cite{Lal2}, \cite{LR1} or \cite{Ru} for more details) and which is reviewed shortly here in as much as it is needed below. We use the notation introduced in Section \ref{sec:3.1}.

Now, consider the operators $\mathcal L_z$ for $z\in \mathbb C^d$ defined analogously to the real case. Then it is known that $z\mapsto \mathcal L_z$ is an entire holomorphic function of $z \in \mathbb C^d$, and
\[ \left(\partial/\partial z_j\right)\mathcal L_zg= \mathcal L_z(f_jg).\] 

Here we recall the formulation by Lalley.
 
\begin{proposition}[\cite{Lal1}, Proposition 4]\label{prop:3.1} The functions $z \mapsto \Lambda_z$,  $z\mapsto \varphi_z$ have analytic extensions to a complex neighborhood $U\subset \mathbb C^d$ of $z_0\in \mathbb R^d$, such that
 \[ \mathcal L_z \varphi_z=\Lambda_z \varphi_z  \;\mbox{ and }\; \int \varphi_z \d\nu_0 =1\qquad z\in U.
 \]
The function $z\mapsto \nu_z$ extends to a weak-$\star$-analytic measure-valued function on $U$ such that
\[ \mathcal L^*_z\nu_z=\Lambda_z \nu_z \;\mbox{ and }\;  \int \varphi_z \d\nu_z=1\qquad z\in U.
\]
For each  $z^*\in \mathbb R^d$ and each $\delta >0$ there exists $\epsilon=\epsilon(\delta, z^*) >0$ such that if $z\in U$ and $\|z - z^*\| \le \epsilon$, then
\[
     \operatorname{spectrum} \mathcal L_z\setminus\{\lambda_z\}\subset \{ x\in \mathbb C:\ |z|\le \Lambda_{z^*}-\delta\}.\] 
 \end{proposition}
     
 Note that weak-$\star$-analytic means that for each H\"older continuous function $g$ on $\Omega$ the map $z \mapsto \int g \d\nu_z$ is analytic.  
 
 For $g:\Theta\to \mathbb R$ and $t\in \overline \Theta$ write
 \begin{eqnarray*}
 && \widetilde g^t= g-\int g \d\mu_t \quad \mbox{\rm (compare Section \ref{sec:2.2})}\\
 && D_i[g(s)](t)= \frac{\partial g(s)}{\partial s_i}\Big|_{s=t}\qquad g\ \mbox{differentiable and }\ 1\le i\le d\\
 && D_{ij}[g(s)](t)= \frac{\partial^2 g(s)}{\partial s_i\partial s_j}\Big|_{s=t}\quad g\ \mbox{twice differentiable,} \ 1\le i,j\le d 
 \end{eqnarray*}
 
 Much more is known about the derivatives of the maps in Proposition \ref{prop:3.1}. For convenience we include a short proof as well
 \begin{lemma}\label{lem:3.2} The map $s\mapsto P(F_s)$, $s\in \Theta$, is analytic and satisfies:
 \begin{enumerate} 
 \item  Its first partial derivatives are
 \begin{equation*}
  D_i[P(F_s)](t) = \int f_i\d\mu_{t},\qquad 1\le i\le d.
  \end{equation*}
\item Its second partial derivatives are
 \begin{equation*}
 D_{ij}[P(F_s)](t)= \lim_{n\to\infty}\frac 1n \int S_n\widetilde  f_i^t S_n\widetilde f_j^t \d\mu_t.
 \end{equation*}
 \item The third order partial derivatives of $s\mapsto P(F_s)$ are bounded.
 \end{enumerate}
 \end{lemma}
 
 \begin{proof} 1. Let $\hat s_i=(0,\ldots ,0,s_{i},0,\ldots 0)\in \mathbb C^d$ with $i$-th coordinate equal to $s_{i}$.  Differentiating each side of the equality 
 \[ \mathcal L_{t+\hat s_i} \varphi_{t+\hat s_i} = \Lambda_{t+\hat s_i}\varphi_{t+\hat s_i}\]
with respect to  $s_{i}$ and evaluating at $s_{i}=0$ yields
 \[ \frac{d}{ds_{i}}[\mathcal L_{t+\hat s_i} \varphi_{t+\hat s_i}](0)= \mathcal L_t\left(\frac{d}{ds_{i}}[\varphi_{t+\hat s_i}](0) +\varphi_{t} f_i\right)\]
 and 
 \[ \frac{d}{ds_{i}}[\e^{P(F_{t+\hat s_i})}\varphi_{t+\hat s_i}](0)= \e^{P(F_t)}\left(\frac{d}{ds_{i}}[P(F_{t+\hat s_i})](0) \varphi_t+ \frac{d}{ds_{i}}[\varphi_{t+\hat s_i}](0)\right).\]
 Equating, multiplying by $\Lambda_t^{-1}$ and integration with $\nu_t$ implies
 \[ \int \frac{d}{ds_{i}}[\varphi_{t+\hat s_i}](0) +\varphi_{t} f_i \d\nu_t=  \int
\frac{d}{ds_{i}}[P(F_{t+\hat s_i})](0) \varphi_t+  \frac{d}{ds_{i}}[\varphi_{t+\hat s_i}](0) \d\nu_t.\]
 This proves the first part.
 
 2. In this case we use $\mathcal L_{t+\hat s_i+\hat s_j}^n\varphi_{t+\hat s_i+\hat s_j} = \Lambda_{t+\hat s_i+\hat s_j}^n \varphi_{t+\hat s_i+\hat s_j}$ to obtain
  \begin{eqnarray*}
  && D_{ij}[\mathcal L_{t+\hat s_i+\hat s_j}^n \varphi_{t+\hat s_i+\hat s_j}](0)= \sum_{k\cdot\in \Omega} D_{ij}[ \varphi_{t+\hat s_i+\hat s_j}(k\cdot) \e^{S_n F_{t+\hat s_i+\hat s_j}(k\cdot)}](0)
 = \\
   && \mathcal L_t^n\left[ D_{ij}[\varphi_{t+\hat s_i+\hat s_j}](0)+D_i[\varphi_{t+\hat s_i}](0) S_nf_j + D_{j}[\varphi_{t+\hat s_j}](0) S_n f_i + S_nf_i S_nf_j\varphi_t\right]
   \end{eqnarray*}
 and 
 \begin{eqnarray*}
 && D_{ij}[\e^{nP(F_{t+\hat s_i+\hat s_j})}\varphi_{t+\hat s_i+\hat s_j}](0)= \e^{nP(F_t)}\{nD_{ij}[P(F_{t+\hat s_i +\hat s_j})](0) \varphi_t \\
 &&\quad + n^2\int f_i d\mu_t\int f_j \d\mu_t \varphi_t + n D_i[\varphi_{t+s_i+s_j}](0) \int f_j \d\mu_t\\
 && \quad+n  \int f_i \d\mu_t D_j[\varphi_{t+s_i+s_j}](0) + D_{ij}[ \varphi_{t+\hat s_i+\hat s_j}](0)\}.
 \end{eqnarray*}
 Integrating with respect to $\mu_t$, using the ergodic theorem, writing $\mu_t(f_i)=\int f_i \d\mu_t$ and letting $n\to\infty$ it follows that
 \begin{eqnarray*}
 D_{ij}[P(F_{t+\hat s_i+\hat s_j})](0)&=&\lim_{n\to\infty}
 \frac 1n \int S_nf_iS_nf_j \d\mu_t - n \mu_t(f_i)\mu_t(f_j)\\
 &=& \lim_{n\to\infty} \frac 1n \int S_n\widetilde f_i^t S_n\widetilde f_j^t \d\mu_t.
 \end{eqnarray*}
 The convergence of the last expression is well known and called the asymptotic covariance of $f_i$ and $f_j$ with respect to $\mu_{t}$ (see Section \ref{sec:2.2}).
 
 3. The maps $z\mapsto \frac 1n \int S_n\widetilde f_i^z  S_n\widetilde f_j^z \d\mu_z$ are bounded and analytic, where we use integration by $\mu_z$ as a notation for the linear functional extending $\int\,\cdot \, \d\mu_t$ analytically. Hence, by the Cauchy  differentiation formula, the partial derivatives of $D_{ij}[P(F_z)]$ are bounded in some sufficiently small neighborhood $U(t_1,\ldots ,t_d)$. 
 \end{proof}
 
 \begin{lemma}\label{lem:3.3} For any cylinder $C=[c_0,\ldots ,c_{n-1}]$ the map $\mathbb C\ni z\mapsto \nu_z(C)$ is analytic such that
  \begin{enumerate}
 \item its partial derivatives are given by
 \begin{equation}\label{eq:3.3}
  D_i[\nu_s(C)](t)=  \int_C (-n \mu_t(f_i)+S_nf_i(x)) \d\nu_t(x) + O(\nu_t(C))\qquad 1\le i\le d,
  \end{equation}
 \item and its second order partial derivatives are for $1\le i,j\le d$
 \begin{equation}\label{eq:3.4}
  D_{ij}[\nu_s(C)](t) =  \int_C S_n\widetilde f_i^t S_n\widetilde f_j^t \d\nu_t  -  \nu_t(C) \int S_n\widetilde f_i^t S_n\widetilde f_j^t \d\mu_t +o(n\nu_t(C))  
  \end{equation}
    \item Higher order partial derivatives are bounded.
  \end{enumerate}
 \end{lemma}
 
 \begin{proof} Analyticity follows from Proposition  \ref{prop:3.1}.
 \begin{enumerate}
 \item  Since $\nu_s(C)=\int \e^{-nP(F_s)+S_nF_s(cx)} \d\nu_s(x)$, for fixed $t\in \mathbb R^d$
 \begin{eqnarray*}
  D_i[\nu_s(C)](t)&=& \int (-n \mu_t(f_i)+S_nf_i(cx))\e^{-nP(F_t)+S_nF_t(cx)} \d\nu_t(x)\\
 && \qquad\qquad\qquad \qquad+ D_i[ \int \e^{-nP(F_t)+S_nF_t(cx)} \d\nu_s(x)](t)\\
  &=& \int_C (-n \mu_t(f_i)+S_nf_i(x)) \d\nu_t(x) \\
  && \qquad\qquad\qquad \qquad+ D_i[ \int \e^{-nP(F_t)+S_nF_t(cx)} \d\nu_s(x)](t).
  \end{eqnarray*}
  We have by (\ref{eq:3.1}) and some constant $0<D(t)<\infty$
 \[ D(t)^{-1} \le \sup_{x\in \Omega} \e^{-nP(F_t)+S_nF_t(cx)}\frac 1{\nu_t(C)}  \le D(t), \] 
 and hence  for $z$ in some neighbourhood of $t$ also
 \[ \left| \int   \e^{-nP(F_t)+S_nF_t(cx)} \d\nu_z(x)\right|\le D(t)\nu_t(C) \| \nu_z\|<\infty,\] 
 where by Proposition \ref{prop:3.1} $\nu_z$ stands for the analytic extension of $\nu_s$, considered as a linear operator. By the Cauchy differentiation formula we then obtain that
 \[ \left| \frac{\partial}{\partial z_i}  \int \e^{-nP(F_t)+S_nF_t(cx)} \d\nu_z(x)\big|_{z=t}\right| \le K\nu_t(C)\] 
 for some constant $K<\infty$ (see the proof of   Proposition   4.4  in \cite{Chu} for this essential detail). 
 
\item This argument - applied to the partial derivative -  also shows that the second partial derivative
 \[ D_{ij}\left(\int \e^{-nP(F_t)+S_nF_t(cx)} \d\nu_s(x)\right)(t)\] 
 is bounded. Therefore, by Lemma \ref{lem:3.2} and using  (\ref{eq:3.3}) and (\ref{eq:3.4}),  
 \begin{eqnarray*}
 &&D_{ij}[\nu_s(C)](t)\\
 &&= D_j\left[\int (-n \mu_s(f_i)+S_nf_i(cx))\e^{-nP(F_s)+S_nF_s(cx)} \d\nu_s(x)\right](t) + O(1) \\
&&= \int S_n\widetilde f_i^t(cx)S_n\widetilde f_j^t(cx) \e^{-nP(F_t)+S_nF_t(cx)} \d\nu_t(x) \\
&& \quad -  n D_{ij}\left[P(F_s)\right](t) \int \e^{-nP(F_t)+S_nF_t(cx)} \d\nu_t(x)\\
&&\quad +D_j\left[\int -S_n\widetilde f_i^t(cx) \e^{-nP(F_t)+S_nF_t(cx)} \d\nu_s(x)\right](t) +O(1)\\
&&=\int_C S_n\widetilde f_i^tS_n\widetilde f_j^t  \d\nu_t -  n D_{ij}\left[P(F_s)\right](t) \nu_t(C)+ o(n\nu_t(C))+O(\nu_t(C)),
\end{eqnarray*}
since by (\ref{eq:3.1})
\[\frac 1{n\nu_t(C)} \int -S_n\widetilde f_i^t(cx) \e^{-nP(F_t)+S_nF_t(cx)} \d\nu_z(x) = o\left(\frac{   \nu_t(C) \|\nu_z\|}{n\nu_t(C)}\right)\] 
 is bounded in some neighborhood of $z=t$  the partial derivative in question is bounded; moreover, the bound of this function tends to zero a.\,s.\@ by the ergodic theorem  as $n\to\infty$, hence the derivative in question   is a.\,s.\@ $o(n\nu_t(C))$.

\item Similar as the proof of 3. in Lemma \ref{lem:3.2}.
\end{enumerate}
 \end{proof}
 
 We continue with  a result, proved  in case $d=1$ in \cite{Chu}.
 
\begin{lemma} \label{lem:3.4} Let $C=[c_0,\ldots ,c_{n-1}]$ be a cylinder of length $n$. The function 
\[ z \in \mathbb R^d\mapsto \log \nu_z(C)\] 
satisfies for every $t\in \mathbb R^d$ and $1\le i\le d$
\[ D_i[\log \nu_s(C)](t)= \frac 1{\nu_t(C)} \int_C S_n\widetilde f_i^t \d\nu_t  + O(1).\] 
\end{lemma}
\begin{proof}
First note that by Proposition \ref{prop:3.1}  for each cylinder $C=[c_0,c_1,\ldots ,c_{n-1}]$ and each $z_0\in \mathbb R$
\[ z\mapsto\nu_z(C)= \int_\Omega \exp(-nP(F_z)+S_n F_z(cx)) d\nu_z(x)\] 
is analytic in some complex neighborhood of $z_0=t$.

Taking the partial derivative and evaluating at $z=t$ yields by Lemma \ref{lem:3.3} that
\begin{eqnarray*}
&&  D_i\left[\int_\Omega  \exp(-nP(F_s)+S_nF_s(cx)) \d\nu_s(x)\right](t)\\
&&=  \int (S_n\widetilde f_i^t(cx) \exp(-nP(F_t)+S_n F_t(cx)) \d\nu_t(x) + O(\nu_t(C))
\end{eqnarray*}
and hence
\begin{eqnarray*}
&&  D_i[\log \int_\Omega  \exp\{-nP(F_s)+S_n F_s(cx)\} \d\nu_s(x)](t)\\
&&=  \frac 1{\nu_t(C)}\int_C S_n\widetilde f_i^t(x)  \d\nu_t(x) + O(1)
\end{eqnarray*}
   \end{proof}
   
 For completeness we conclude this section with  a more detailed proof of Theorem \ref{theo:2.6}.
   \begin{lemma} \label{lem:3.5} Let $\mu$ be an invariant Gibbs measure.  The family of H\"older potentials $(f_{1},\ldots , f_{d})$ is independent as cohomology classes if and only if the corresponding  covariance matrix $\Sigma^{\mu}$ is positive definite. 
\end{lemma}
\begin{proof}
For $s\in \R^{d}\setminus \{0\}$ consider $f_{s}= \sum s_{i}f_{i}$ and suppose it is cohomologous to $0$, i.\,e.\@ there exists a continuous function $h$ such that   $f_{s}=c+h-h\circ \sigma$. It follows that $\int f_{s} \d\mu=c$ and therefore 
\begin{eqnarray*} s\Sigma^{\mu } s^{t}&=&\lim 1/n \int S_{n}(f_{s}-\mu(f_{s}))S_{n}((f_{s}-\mu(f_{s})) \d\mu\\
&=& \lim 1/n\int S_{n}(h-h\circ \sigma)S_{n}(h-h\circ \sigma)\d\mu\\
&=&\lim 1/n\int (h-h\circ \sigma^{n})^{2 }\d\mu =0.\end{eqnarray*}
 Therefore, the covariance matrix is not positive definite. 

On the other hand, if the covariance matrix is not positive definite, then there exists $s\in \R^{d}\setminus \{0\}$ such that $s\Sigma^{\mu } s^{t}=0$. Using the spectral gap properties of the Perron-Frobenius operator with respect to the normalized potential for the Gibbs measure $\mu$, we find that for $\tilde f^{\mu}_{s}$  there is another H\"older continuous function $h$ such that $\tilde f^{\mu}_{s}=h-h\circ\sigma$ and   $s\Sigma^{\mu } s^{t}=\Cov_{\mu}(f_{s})=\int h^{2} \d\mu$. So if this variance vanishes, then  $h=0$ and $f_{s}$ is cohomologous to $0$. 
   \end{proof}

  \subsection{Multivariate statistical calculus for Gibbs measures}\label{sec:3.3}
  
  Denote by $\mathcal F$ the Borel $\sigma$-field on $\Omega$ and let $\mu$ be a   shift-invariant  Gibbs measure on $\mathcal F$. Let $\mathcal F_k^l$ be the $\sigma$-field generated by $\{X_j: k\le j <l\}$ for  $0\le k < l\le  \infty$, and  let $g_i$, $0\le i\le d$, be finitely many H\"older continuous functions. Hence, for some $\rho<1$, 
  \[ |g_i(x)-g_i(y)|\le H_i \rho^n\qquad 0\le i\le d, \ x_j=y_j \ (0\le j<n),\] 
  where $H_i$ denotes the H\"older-norm of $g_i$. Define
  \[   g_i^n(x) = \frac1{\mu([x_0,\ldots ,x_{n-1}])}\int_{[x_0,,,.x_{n-1}]} g_i(y) \d\mu(y).\] 
  Then $|g_i(x)-g_i^n(x)|\le H_i\rho^n$.
  
 Recall that a Gibbs measure $\mu$ on a topologically mixing subshift of finite type is $\psi$-mixing with exponentially fast decaying mixing coefficients $\psi(n)$, $n\ge 1$, that is
  \begin{equation}\label{eq:3.5} \sup_{A\in \mathcal F_0^k; B\in \mathcal F_{k+n}^\infty}\left| \frac{\mu(A\cap B)-\mu(A)\mu(B)}{\mu(A)\mu(B)}\right|\le \psi(n)
  \end{equation}
  and $ \psi(n)\le K \kappa^n$ for some constants $K>0$ and $0<\kappa <1$.
  
As before, we shall write $\mu(g):=\int g \d\mu$.  We first recall some known facts adapted to Gibbs measures, beginning with covariances of functions  under $\mu$ compared to covariances of the same functions under product measures. This is essentially from \cite{DK1}.
 
 \begin{lemma}\label{lem:3.6} Let $\mu$ be a Gibbs measure with mixing coefficients $\psi(n)$. Assume that $0\le k_1<k_3$, $0\le k_2$ and $k_2\vee k_3\le k_4$. Define $n= \min\{ |k_j-k_i|: 1\le i\ne j \le 4\}$. Writing $n=p+q$, then for $A_i \in\mathcal F_{k_i}^{k_i+p}$, $1\le i\le 3$ and $A_4\in \mathcal F_{k_4}^\infty$
 \[ \frac{|\mu(A_1\cap A_3\cap A_2\cap A_4)- \mu(A_1\cap A_3)\mu(A_2\cap A_4)|}{\mu(A_1\cap A_3)\mu(A_2\cap A_4)} = O(\psi(q)).\] 
 \end{lemma}
 \begin{proof} Consider $k_1\le k_2\le k_3< k_4$. The other cases are similar. The claim follows from (\ref{eq:3.5}) and { an iterated application of the triangle inequality splitting off $A_4$, $A_3$ and $A_2$ successively, then combing $A_1$ and $A_3$ and finally $A_1$ and $A_4$:
 \begin{eqnarray*}
 &&\!\!\!\!\!\!\!\!\!\!\!|\mu(A_1\cap A_3\cap A_2\cap A_4)- \mu(A_1\cap A_3)\mu(A_2\cap A_4)|\\
&\leq&[ \psi(q)(1+\psi(q))^2+ \psi(q)((1+\psi(q))+1)] \mu(A_1)\mu(A_2)\mu(A_3)\mu(A_4) \\
&& \;\;\;\;\;\;\;\; + \psi(q)(1-\psi(q))^{-1}\mu(A_2\cap A_4)\mu(A_1\cap A_3) \\
&\leq&\{ [ \psi(q)(1+\psi(q))^2+ \psi(q)((1+\psi(q))+1)] (1-\psi(q))^{-2}\\ && \;\;\;\;\;\;\;\;+\psi(q)(1-\psi(q))^{-1})\}\mu(A_2\cap A_4)\mu(A_1\cap A_3)\\
&=&  O(\psi(q))\mu(A_1\cap A_3)\mu(A_2\cap A_4).
 \end{eqnarray*}}
 \end{proof}
  \begin{lemma}\label{lem:3.7} Let $n=p+q$, $0\le k_1<k_3$ and $0\le k_2\le k_4$ be as in Lemma \ref{lem:3.6}. Then for H\"older continuous functions $g_i:\Omega\to\mathbb R$, $i=1,2,3$,
  \begin{eqnarray}\label{eq:3.6}  && \lvert\mu (g_1\circ \sigma^{k_1} g_2\circ \sigma^{k_3} g_1\circ \sigma^{k_2} g_3\circ \sigma^{k_4}) - \mu(g_1\circ \sigma^{k_1} g_2\circ \sigma^{k_3})\mu(g_1\circ \sigma^{k_2} g_3\circ \sigma^{k_4})\rvert \notag\\
  && \qquad\le  \sqrt{2 K \psi(q)} (\|g_1\|_\infty+\|g_2\|_\infty+\|g_3\|_\infty) + 8 \max\{H_1,H_2\} \rho^p,
  \end{eqnarray}
  where $K$ is some constant independent of $g_1$ and $g_2$
  \end{lemma}
  \begin{proof}  Replacing the functions $g_i$ and $g_j$ in (\ref{eq:3.6}) by their approximations $g_i^p$ and $g_j^p$ the error in the left hand side is at most $8 \max\{H_i,H_j\} \rho^n$. Then apply Lemma 3 b) in \cite{DK1} together with Lemma \ref{lem:3.6}.
   \end{proof}
  
  Apply Lemma \ref{lem:3.7} to the functions $g_1=f_i-\mu_t(f_i)$, $g_2=f_j-\mu_t(f_j)$ and $g_3= \sum_{l=0}^L (f_j-\mu_t(f_j))\circ \sigma^l$ to obtain

  \begin{corollary}\label{cor:3.8} For the potentials $f_i$ ($0\le i\le d$) in Theorem \ref{theo:2.8}
  \begin{eqnarray*}
  &&\int \widetilde f_i^t\circ \sigma^{k_1}\widetilde f_j^t\circ \sigma^{k_3}\widetilde f_i^t\circ \sigma^{k_2}\sum_{l=0}^L \widetilde f_j^t\circ \sigma^{k_4+l} \d\mu_t = \\
  && =
  \int \widetilde f_i^t\circ \sigma^{k_1}\widetilde f_j^t\circ \sigma^{k_3}\d\mu_t \int \widetilde f_i^t\circ \sigma^{k_2}\sum_{l=L} \widetilde f_j^t\circ \sigma^{k_4+l} \d\mu_t +\\
  && \qquad\qquad +O\left((\psi(p)^{1/2}+\rho^p)(\|g_1\|+L\|g_2\|)\right).
  \end{eqnarray*}
  \end{corollary}
This corollary may be used to calculate the convergence of the covariance matrices in this note, especially in Sections \ref{sec:2.3} and \ref{sec:3.4}.

  Next we consider a function $h:\Omega\times \Omega \to\mathbb R$ 
  of the form
  \begin{equation}\label{eq:3.7}
  h(x,y)= f(x)g(y)+f(y)g(x) \qquad x,y\in \Omega.
  \end{equation}
   We assume that $\int f \d\mu=\int g \d\mu=0$ and we notice that $h$ is symmetric, i.\,e.\@ $h(x,y)=h(y,x)$. This is the general case  of a bivariate kernel since for any function $k(x,y)=f(x)g(x)$ the function
  \[ h(x,y) = \frac 12 (f(x)g(y)+f(y)g(x))\]  
  leads to the same von Mises functional
  \[ M_n(k)(x,y)=\sum_{k,l=0}^{n-1} k(\sigma^k(x),\sigma^l(y)) = \sum_{k,l=0}^{n-1} h(\sigma^k(x),\sigma^l(y)).\] 
  Such statistical functionals will be considered next. We shall apply the theory developed in \cite{DG}. Such applications as here are sketched in \cite{DG} but need to be further specialized in the present context of Gibbs measures.
  
  \begin{proposition}\label{prop:3.9} Let  $h$ be as in (\ref{eq:3.7}) and in $L_2\coloneqq L_2(\mu\times\mu)$. Then the operator
  \[Tu(x)=\int h(x,y) u(y) \d\mu(y)\] 
  is a self-adjoint Hilbert-Schmidt  operator on $L_2$ with eigenvalues 
 \begin{equation}\label{eq:3.8}
  \lambda_1=\int fg \d\mu+\|f\|_{L_2}\|g\|_{L_2}\quad
  \lambda_2=\int fg \d\mu- \|f\|_{L_2}\|g\|_{L_2}
  \end{equation} 
  The corresponding orthonormal eigenfunctions are
   \begin{equation}\label{eq:3.9}
  \phi_{i}(x)=  c_i \left(\frac f{\|f\|_{L_2}^2\|g\|_{L_2}} -(-1)^i \frac g{\|f\|_{L_2}\|g\|_{L_2}^2}\right) \quad i=1,2
   \end{equation}
   with norming constants $c_1$ and $c_2$.
 \end{proposition}
 
 \begin{proof} This is well known. Since
 \begin{eqnarray*}
  Tu(x) &=& \int (f(y)g(x)+f(x)g(y))u(y)\d\mu(y)\\
  &=&f(x)\int g(y)u(y)\d\mu(y)+g(x) \int f(y)u(y)\d\mu(y)
  \end{eqnarray*}
 the range of $T$ is 2-dimensional. The eigenvalues  are computed from the equations
 \begin{eqnarray*}
&&  a\int fg\d\mu +b\int g^2\d\mu= \lambda a\\
&& a\int f^2\d\mu +b\int fg\d\mu= \lambda b.
\end{eqnarray*} 
Thus,
\begin{equation*}
 \lambda_{1,2}=\int fg \d\mu\pm \|f\|_{L_2}\|g\|_{L_2}.
 \end{equation*}
 One easily checks that  eigenfunctions are
 \begin{equation*}
 \phi_{1,2}(x)= \left(\|\frac  f{\|f\|_{L_2}^2\|g\|_{L_2}} \pm \frac g {\|f\|_{L_2}\|g\|_{L_2}^2}\|_{L_2}\right)^{-1} 
 \left(\frac f{\|f\|_{L_2}^2\|g\|_{L_2})} \pm \frac g{\|f\|_{L_2}\|g\|_{L_2}^2}\right).
 \end{equation*}
Note that $\phi_1$ and $\phi_2$ are orthonormal.
 \end{proof} 
  
  \begin{theorem}\label{theo:3.10} Let $\mu$ be a an invariant Gibbs measure with H\"older continuous potential $F$ on a  topologically mixing subshift of finite type $\Omega$.   Let $f,g:\Omega\to\mathbb R$ be H\"older continuous and centered. Then
  $\frac 1n S_nfS_ng$ is a von Mises functional with kernel $k(x,y)=f(x)g(y)$ and has a representation
  \[ \frac 1n S_nfS_ng=\frac 1{n}\left( \lambda_1 (S_n\phi_1)^2 +\lambda_2 (S_n\phi_2)^2\right),\] 
  where $\lambda_i$ and $\phi_i$ are the eigenvalues (\ref{eq:3.8}) and orthonormal eigenfunctions (\ref{eq:3.9}) for the kernel $h(x,y)=\frac 12(k(x,y)+k(y,x))$.
  
Moreover, 
\[ \frac 1n S_nfS_ng\]    
  converges weakly to the distribution of $\lambda_1 Z_i+\lambda_2 Z_2$ where $Z_i$ are independent and have the distribution of the square of a standard normal random variable.
  \end{theorem}
  
  \begin{proof} By the remark after (\ref{eq:3.7}) 
  \[ M_n(k)= M_n(h)= S_nfS_ng.\] 
  By Proposition \ref{prop:3.9} and the spectral representation of  Hilbert-Schmidt operators
  \[ h(x,y) =  \lambda_1 \phi_1(x)\phi_1(y)+  \lambda_2 \phi_2(x)\phi_2(y).\] 
  The proof is completed by applying Theorem 4 in \cite{DG}. Note that this theorem applies for Gibbs measures on topologically mixing subshifts of finite type since $(X_k)_{k\ge 0}$ is $\psi$-mixing with summable rates (see the discussion in \cite{DG}, Section 9.2.2 ({\bf b})).
  \end{proof}
  
  \begin{remark}\label{rem:3.11} 
  Note that by its representation
  \[ \frac 1n S_nf(x)S_ng(x)= \frac 1{n} \left( \lambda_1 (S_n\phi_1(x))^2 +\lambda_2 (S_n\phi_2(x))^2\right).\] 
  Since $\frac 1{\sqrt{n}} S_n\phi_i$ are asymptotically normal and independent, the last statement in Theorem \ref{theo:3.10} follows.
  \end{remark}

\begin{corollary}\label{cor:3.12} Let $\mu_t$ denote invariant Gibbs measures for the potential $F_t=t_1f_1+t_2f_2$ where $t\in \mathbb R^2$. Let $\nu_t$ be the associated Gibbs measure for $\mathcal L_{F_t}$. Then for $i,j\in \{1,2\}$ and $\theta\in \mathbb R^2$
\[ \frac 1{n\nu_\theta([x_0,\ldots ,x_{n-1}])} D_{ij} [\nu_t([x_0,\ldots ,x_{n-1}])](\theta)\] 
converges weakly to the distribution $\lambda_1 Z_1^2+\lambda_2  Z_2^2 -\Sigma_{ij}$
where $Z_k$ are independent normal distributions, $\lambda_k$ and $\phi_k$ are the eigenvalues and eigenfunctions derived in Proposition \ref{prop:3.9} and where 
\[ \Sigma_{ij}^{\mu_{\theta}}=\lim_{m\to\infty} \frac 1m \int   S_m\widetilde f_i^\theta S_m\widetilde f_j^\theta \d\mu_\theta \] 
defines the asymptotic covariance matrix of $(\frac 1{\sqrt{n}} S_n\widetilde f_i^\theta)_{1\le i\le 2}$ under the invariant Gibbs distribution $\mu_\theta$.
\end{corollary}
\begin{proof} By Lemma \ref{lem:3.3} 
\begin{eqnarray*}
&&  \frac 1{\nu_\theta([x_0,\ldots ,x_{n-1}])} D_{ij} [\nu_t([x_0,\ldots ,x_{n-1}])](\theta)\\
&& = \frac 1{\nu_\theta([x_0,\ldots ,x_{n-1}])} \int_{[x_0,\ldots ,x_{n-1}]} S_n\widetilde f_i^\theta S_n\widetilde f_j^\theta \d\nu_\theta - \int  S_n\widetilde f_i^\theta S_n\widetilde f_j^\theta \d\mu_\theta + o(n).
\end{eqnarray*}
Hence \[\frac 1{n\nu_\theta([x_0,\ldots ,x_{n-1}])} D_{ij} [\nu_t([x_0,\ldots ,x_{n-1}])](\theta)\]  and 
\[\frac 1n S_n\widetilde f_i^\theta S_n\widetilde f_j^\theta
-\lim_{m\to\infty} \frac 1m \int \int  S_m\widetilde f_i^\theta S_m\widetilde f_j^\theta \d\mu_\theta\] 
are asymptotically stochastic equivalent, so have the same limiting distribution.
By Theorem \ref{theo:3.10} the result follows.
\end{proof}

\subsection{Cram\'er--Rao bound}\label{sec:3.4}

In this section we   let $d,m\in \mathbb N$ and consider a  family
\[ \psi^n:\Theta\subset \mathbb R^d\to \mathbb{R}^{m},\qquad n\ge 0\] 
of functions of the parameter space.
We reformulate the classical Cram\'er--Rao bound for Gibbs distributions (\cite{Rao}, \cite{Cr}). We assume that $\Theta$ is an open set and $\Psi^n$ is differentiable on $\Theta$ for every $n\ge 1$.  Since the densities for the MLE are only given for the Gibbs measures $\nu_\theta$ explicitly we  use the Cram\'er--Rao  bounds only for these densities. For completeness we formulate the results because the applications in Sections \ref{sec:5} and \ref{sec:6}, and the connection to Theorem \ref{theo:2.8} are related to these bounds.

\begin{theorem}[Cram\'er--Rao Bound]\label{theo:3.13}
Let $\{T_n: n\ge 1\}$ be  a family of unbiased $\mathbb R$-valued (i.\,e.\@ $d=1$) statistics where $T_n$ estimates the parameter $\psi^n(\theta)=\mathbb{E}_{\theta}T_n(X_0,\ldots ,X_{n-1})$  in the family $(\nu_\theta^n:n\ge 1)$ of marginal distributions on $\mathcal F_0^n$ as explained in in Subsection \ref{sec:3.1}. Then for any $\theta_0\in \Theta$
\begin{eqnarray}\label{eq:3.10}
&& \int \left(T_n(x_0,\ldots ,x_{n-1})-\psi^n(\theta_0) \right)^2\d\nu_{\theta_0}^{n}((x_0,\ldots ,x_{n-1})) \cdot \notag \\
&& \quad \cdot \frac{1}{n}
\int \left(\frac{d}{d\theta }|_{\theta=\theta_0} \log \d\nu_\theta^n(x_0,\ldots ,x_{n-1})\right)^2 \d\nu_{
\theta_0}^n ((x_0,\ldots ,x_{n-1})) \notag\\
&&\qquad \geq 
\,\left(\frac{d}{d\theta }|_{\theta=\theta_0}   \psi^n(\theta)\right)^2.
\end{eqnarray}
\end{theorem}
\begin{proof} Fix $n\ge 1$. In order to estimate $\psi^n(\theta)$ in the family $\{\nu_\theta^n: \theta\ge 0\}$ of probabilities $\nu_\theta^n$ on $\{1,\ldots,a\}^n$ note that the dominating measure can be chosen as the counting measure so that the density becomes $(x_1,x_2,\ldots ,x_n)\mapsto \nu_\theta^n(x_0,\ldots ,x_{n-1}):=\nu_\theta([x_1,\ldots ,x_n])$. Then, in view of Proposition \ref{prop:3.1}, apply the classical Cram\'er--Rao theorem (see \cite{Cr} and \cite{Rao}).
\end{proof}

Likewise we can reformulate the Cram\'er--Rao bound in the multivariate case  $d>1$. The Fisher information $I(\theta)$ in the multivariate case is defined by
\begin{eqnarray}\label{eq:3.11}
&& \qquad I^n(\theta)= \left( I_{k,l}^n(\theta)\right)_{1\le k,l\le d}\\
I_{k,l}^n(\cdot)&=& E\left[\frac {\partial}{\partial \theta_k}\log \nu_\theta^n([X_0,\ldots ,X_{n-1}]) \frac {\partial}{\partial \theta_j}\log \nu_\theta^n([X_0,\ldots ,X_{n-1}]) \right] \notag \\
&=& - E\left[ \frac {\partial^2}{\partial \theta_k\partial \theta_j}\log \nu_\theta^n([X_0,\ldots ,X_{n-1}])\right]\notag
\end{eqnarray}
\begin{theorem}\label{theo:3.14}
Let $\{T_n: n\ge 1\}$ be  a family of unbiased $\mathbb R^m$-valued  statistics where $T_n$ estimates the parameter $\psi^n(\theta)=\E T_n(X_0,\ldots ,X_{n-1})$  in the family $(\nu_\theta^n:n\in\N)$  as in Subsection \ref{sec:3.1}. Then for any $\theta_0\in \Theta$ the matrix
\[ 
\Cov _{\theta_0}(T^n(X_0,\ldots ,X_{n-1}))- \left(\frac{\partial \psi_j^n(\theta)}{\partial \theta_i}(\theta)\right)_{1\le i,j\le d}I^n(\theta_0)^{-1}\left(\frac{\partial \psi_j^n(\theta)}{\partial \theta_i}(\theta)\right)_{1\le i,j\le d}^t
\] 
is positive semidefinite,
where
$ \Cov _{\theta_0}(T_n)$
denotes the covariance matrix of  $T_n$.
\end{theorem}
\begin{remark}\label{rem:3.15}
The theorem says in particular that the $i$-th coordinate of $T_n$, $T_n^i$, satisfies 
\[ \E_\theta(T_n^i-\psi_i^n(\theta))^2\ge A_i(\theta)\] 
where $A_i(\theta)$ denotes the $i$-th diagonal element of the lower bound matrix
\[\left(\frac{\partial \psi_k^n(\theta)}{\partial \theta_l}(\theta)\right)_{1\le k,l\le d}I^n(\theta)^{-1}\left(\frac{\partial \psi_k^n(\theta)}{\partial \theta_l}(\theta)\right)_{1\le k,l\le d}^t.\] 
\end{remark}
\begin{definition} \label{def:3.16} A sequence of $T_n$ statistics is called asymptotically efficient at $\theta_0\in \Theta$ if 
\[\lim_{n\to\infty} \E_{\theta_0}\| T_n - \psi^n(\theta_0)\|^2 = \|(A_i(\theta_0))_{1\le i\le d}\|^2\] 
where $A_i$ is as in Remark \ref{rem:3.15}.
\end{definition}

\begin{proposition}\label{prop:3.17}
The Fisher information matrix $I^n(\theta)=(I_{ij}^n(\theta))_{1\le i,j\le d}$ satisfies the relation
\[ \lim_{n\to\infty}\frac 1n I_{ij}^n(\theta) = \lim_{m\to\infty} \frac 1m \int S_m\widetilde f_i^\theta S_m\widetilde f_j^\theta \d\nu_{\theta}\] 
if the distribution of $(X_k)_{k\ge 0}$ is drawn from $\nu_\theta$\end{proposition}

\begin{proof} By Lemma \ref{lem:3.4} for $x\in \Omega$ 
\[ S_n\widetilde f_i^\theta (x) = D_i[\log \nu_s([x_0,\ldots ,x_{n-1}])](\theta) + O(1).\] 
Therefore
\[ \frac 1n I_{ij}^n(\theta)= \frac 1n \int S_n\widetilde f_i^\theta(x)S_n\widetilde f_j^\theta(x)  \d\nu_\theta(x)+O(1/n).\] 
 Applying the $\psi$-mixing property for Gibbs measures this proves the claim.
\end{proof}

Similar results hold of corse for the invariant Gibbs measures $\mu_\theta$; however, the densities are not explicitly given, as we remarked before.

\section{Proof of the main theorems}\label{sec:4}

We use the definitions and notations in Sections \ref{sec:2} and \ref{sec:3} and consider, for each $n\in \mathbb N$, the families 
\[  \Theta\subset \mathbb R^d;\quad \nu_\theta^n \quad \mbox{and}\quad \X_{n}=X_0,\ldots ,X_{n-1}.\]

\begin{lemma}\label{lem:4.1} There exist $\eta_n>0$, $\lim_{n\to\infty} \eta_n=0$ such that $\mu_\theta$ a.\,s.\@ for all sufficiently large $n$ an $\eta_n$-maximum likelihood estimator $\widehat \theta_n= \widehat \theta_n(\X_{n})$  exists so that
\[ \mu_{\widehat \theta_n}\in \Theta_n^{\eta_n}(\X_{n}).\] 
\end{lemma}

\begin{proof}
Let $\eta_n>0$ be a sequence converging to zero and determined below by the large deviation property of H\"older continuous functions. Let $\X=(X_n)_{n\ge 0}$ be drawn from $\nu_\theta$. Then by the ergodic theorem and H\"older continuity
\[\int f_i \d\mu_\theta  = \lim_{n\to\infty} \frac 1n \inf_{x\in [\X_{n}]} S_nf_i =  \lim_{n\to\infty} \frac 1n \sup_{x\in [\X_{n}]} S_nf_i,
\]
hence $\mu_\theta\in \Theta_n^{\eta_n}([\X_{n}])$ for all sufficiently large $n$. In order to see this apply the large deviation estimate (see e.g. \@ \cite{D2}, \cite{DK01}, \cite{L3}, \cite{Lo1})
\[ 
\mu_\theta\left(\left\{\left|\frac 1nS_nf_i -\int f_i \d\mu_\theta \right|\ge \eta_n^2\right\}\right)\le \mathrm{e}^{-n I_i(\eta_n^2) }
\]
for each $0\le i\le d$, where each $I_i$ denotes the information function in the large deviation result for $(f_i)_{i=0,\ldots, d}$.
Then choose $\eta_n\to 0$ such that for each $i$ we get $\sum_{n=1}^\infty \e^{-nI_i(\eta_n^2)}<\infty$ and apply the Borel-Cantelli lemma. 
It follows that 
\[\nu_\theta([\X_{n}]) \le \sup_{t\in \Theta_{n}^{\eta_{n}}(\X_{n})} \nu_t([\X_{n}]).\]  
The map $t\mapsto \nu_t([\X_{n}])$ is continuous by Proposition \ref{prop:3.1}, the set $\{t\in \Theta: \mu_t\in \Theta_n^{\eta_n}(\X_{n})\}$ is bounded and closed, hence $t\mapsto \nu_t([\X_{n}])$ attains its maximum on $\{t\in \Theta\cap[-\eta_n^{-1},\eta_n^{-1}]: \mu_t\in \Theta_n^{\eta_n}(\X_{n})\}$. 
\end{proof}

\begin{proposition}\label{prop:4.2} Let  $\widehat \theta_n$ denote the $\eta_n$-maximum likelihood estimator in Lemma \ref{lem:4.1}. Then for almost all  observation $\mathbb X=(X_l)_{l\in \mathbb N_0}$  drawn from the distribution $\nu_\theta$, we have that \[ \slim_{n\to\infty} \widehat \theta_n= \theta\in \overline \Theta.\]\end{proposition}

\begin{proof} 
 Observe that by the almost sure properties obtained in Lemma \ref{lem:4.1} we have
\[ \left\vert\int F_\theta \d\mu_{\hat\theta_n}-\int F_\theta \d\mu_\theta\right\vert \leq (1+\vert \theta\vert_{1})
 \eta_n^2= O(\eta_n)\]
and similarly,  replacing $F_\theta$  by $F_{\hat\theta_n}$,  
\[ \left\vert\int F_{\hat\theta_n} \d\mu_{\hat\theta_n}-\int F_{\hat\theta_n} \d\mu_\theta\right\vert \leq (1+\vert \hat\theta_{n}\vert_{1})
 \eta_n^2= O(\eta_n).\]
Denote by $h_\mu$ the entropy of the shift-invariant probability $\mu$. The above estimate together with the variational principle  gives
\begin{eqnarray*} 
P(F_\theta)=h_{\mu_\theta}+ \int F_{ \theta} \d\mu_{\theta}
\geq h_{\mu_{\widehat \theta_n}}+ \int F_\theta \d\mu_{\widehat \theta_n} \ge 
h_{\mu_{\widehat \theta_n}} + \int F_\theta \d\mu_\theta - O(\eta_n)
\end{eqnarray*}
and in the same way
\begin{eqnarray*} 
h_{\mu_{\hat\theta_{n}}}+ \int F_{\hat \theta_{n}} \d\mu_{\hat\theta_{n}}
\geq h_{\mu_{\theta}}+ \int F_{\hat \theta_{n}} \d\mu_{\theta} \ge
h_{\mu_{\theta}}+ \int F_{\hat \theta_{n}} \d\mu_{\hat\theta_{n}} - O( \eta_n).
\end{eqnarray*}
Combining both inequalities provides  $\lim h_{\mu_{\widehat \theta_n}}= h_{\mu_{ \theta}}$.
Let $\mu$ be a weak  accumulation point of $(\mu_{\widehat \theta_n})$. Since entropy is upper semicontinuous for the topologically mixing shift space  we find
\[ h_{\mu}+ \int F_\theta \d\mu \ge  h_{\mu_\theta}+ \int F_\theta \d\mu_\theta= P(F_\theta)\] 
and by the uniqueness of equilibrium states and the variational principle we conclude $\mu=\mu_\theta$.  
\end{proof}

\begin{lemma}\label{lem:4.3} For  a symmetric positive-definite matrix $K\in\mathbb{R}^{d\times d}$,   we  set 
\[ G:\mathbb R^d\to \R ^{d\times d},\; G(M)\coloneqq M^{t} M- K. 
\]
Then  the matrix $G(M)$ is  invertible  for  Lebesgue almost every  $M \in \R^{d}$.  \end{lemma}
\begin{proof}
 By assumption the matrix $K$ has an orthonormal system of eigenspaces $E_{i}$ with associated positive eigenvalues $\kappa_{i}>0$, $1\leq i\leq \ell $ for some $\ell\leq d$.
 
  If   \begin{eqnarray*}M \not\in L\coloneqq \bigcup_{i} \left\{u\in E_{i}: \|u\|^{2}\neq \kappa_{i}\right\},\end{eqnarray*}
   then on the one hand,  for $v\in \R M \setminus \{0\}$ we either have  $vK\not\in \R M$ and hence $v(M^{t} M-K) =\langle M,v\rangle M - v K   \neq 0$, or  $vK=\alpha M$
    for some  $\alpha\neq0$, whence $M$ is a eigenvector and there is $i$ with $M\in E_i$. Since by assumption $\|M\|^{2}\neq \kappa_{i}$  and writing $v=\beta M$ it follows that 
    $v(M^{t} M-K) =\langle M,v\rangle M - vK =\beta (\|M\|^{2}-\kappa_{i})M  \neq 0$.
    
 On the other hand,  for $v\in (\R M)^{\perp} \setminus \{0\}$ we have $ v(M^{t} M-K) =- vK  \neq 0$.
 
Hence, $G(M)$ can only be non-invertible if $M\in L$, but $L$ is a  Lebesgue null set. 
\end{proof}

\begin{proposition}\label{prop:4.4} For $\theta\in \Theta$, we assume that the $\eta_n$-maximum likelihood estimators $\widehat \theta_n$ converge to $\theta$ a.\,s.\@ with respect to $\nu_\theta$. 

Then the random vectors 
\[ \left(\frac 1{\sqrt{n}} S_n\tilde f_i^\theta\right)_{1\le i\le d}\] 
converge weakly with respect to $\mu_\theta$ to a centered normal variable $N=(N_1,\ldots ,N_d)$ with covariance  $\Sigma^{\mu_{\theta}}$ defined as in Theorem \ref{theo:2.8}, which may be degenerate. 

Moreover, if $f_1,\ldots,f_d$ are linearly independent as cohomology classes, then
$\sqrt{n}(\widehat \theta_n -\theta)$ converges in distribution to the distribution of 
\[ G^{-1}(N)N^t,\] 
where $G$ is defined in Lemma \ref{lem:4.3} with $K=\Sigma^{\mu_{\theta} }$ and where $ G^{-1}(N)$ denotes the inverse of $ G(N)$. 
\end{proposition}
\begin{proof} The first assertion is standard since by assumption in Section \ref{sec:2} the functions $f_1$,\ldots ,$f_d$ are H\"older continuous.

Since $\widehat \theta_n$ converges to $\theta\in \Theta$ a.s., we may assume that for $\epsilon>0$ there is a set $\Omega_0\subset \Omega$ so that $\mu_\theta(\Omega_0)>1-\epsilon$ and $n_0\in \mathbb N$ so that $\widehat \theta_n(x_0,\ldots ,x_{n-1})\in \Theta$ fo all $n\ge n_0$ all $x=(x_k)_{k\ge 0}\in\Omega_0$. We also assume that $\epsilon$ is small enough to that in the $\epsilon$-ball around $\theta$ all functions considered below  have an  analytic extension to the complex $\epsilon$-ball in $\mathbb C^d$ around $\theta$.

In order to condense the notation,
for a cylinder set $C$ and $\lambda=(t_i^{k_i})_{1\le i\le d}$, $(k_i\ge 0)$, write $|\lambda|= k_1+\ldots +k_d$ and
\[  D_{\lambda}\nu_t(C)= \frac{\partial^{|\lambda|}\nu_t(C)}{\partial^{k_1} t_1\ldots \partial^{k_d}t_d}.\] 
If $|\lambda|=1$ and $k_i=1$ we also write $D_i$ for $D_\lambda$.

Since $\nu_t([x_0,\ldots ,x_{n-1}])$ is analytic by Proposition \ref{prop:3.1} we also may assume that $\widehat \theta_n(x_0,\ldots ,x_{n-1})$ is contained in the domain of the Taylor expansion of $\nu_t([x_0,\ldots ,x_{n-1}])$ at $t=\theta$: Writing $C=[x_0,\ldots ,x_{n-1}]$ and $\delta=(\delta_i)_{1\le i\le d}$ we obtain
\[ \nu_{\theta+\delta}(C)= \nu_\theta(C) + \sum_{|\lambda|=1}^\infty D_\lambda[\nu_t(C)](\theta)\prod_{j=1}^d \delta_j^{k_j}.\] 
Likewise, each partial derivative $D_i[\nu_t(C)]$ has the Taylor  polynomial of order 1
\[ D_i[\nu_t(C)](\theta+\delta)= D_i[\nu_t(C)](\theta)+\sum_{|\lambda|=1}^\infty D_\lambda [D_i[\nu_{t}(C)]](\theta)\prod_{j=1}^d \delta_j^{k_j} + R(\delta)\] 
with remainder term
\[ R(\delta)= \sum_{|\lambda|=2} \frac{D_\lambda[D_i[\nu_t(C)]](\theta)}{k_1!\cdot\ldots \cdot k_d!}\prod_{j=1}^d t_j(\delta)^{k_j}\] 
for some $t(\delta)=(t_j(\delta))_{1\le j\le d}$ in the domain of the Taylor expansion.
Since $\nu_{\widehat \theta)n(x)}$ is maximal  
\[ D_i[\nu_t(C)](\widehat \theta_n) =0\qquad \forall 1\le i\le d\] 
hence for $\delta=\widehat \theta_n-\theta$ and $1\le i\le d$
\[ \sum_{|\lambda|=1} D_\lambda[D_i[\nu_t(C)]](\theta)\prod_{j=1}^d(\widehat \theta_n-\theta)_j^{k_j}=-D_i[\nu_t(C)](\theta)   +O(\|\widehat \theta_n-\theta\|^2),\] 
or 
\[(\widehat \theta_n-\theta)\left(D_{ji}[\nu_t(C)](\theta)\right)_{1\le i,j\le d} = -(D_i[\nu_t(C)](\theta))_{1\le i\le d})^t.\]  
By Lemma \ref{lem:3.3} (\ref{eq:3.3})  we have, for all  $x\in C$,
\[ \frac 1{\sqrt{n}\nu_t(C)} (D_i[\nu_t(C)](\theta))_{1\le i\le d})= \frac 1{\sqrt{n}} ( S_n\widetilde f_i^t(x))_{1\le i\le d}+ o(1)\] 
and by Lemma \ref{lem:3.3} (\ref{eq:3.4}), for all $x\in C$,
  \begin{eqnarray*}
 && \frac 1{n\nu_t(C)} (\left(D_{ji}[\nu_t(C)](\theta)\right)_{1\le i,j\le d}\\
 && \qquad =
\frac 1n\left( S_n\widetilde f_i^\theta(x)S_n\widetilde f_j^\theta(x)-\lim_{m\to\infty} \frac 1m \int S_m\widetilde f_i^\theta S_n\widetilde f_j^\theta\d\mu_\theta\right)_{1\le i,j\le d}+ o(1).
\end{eqnarray*}
Now,
$ \Sigma^{\mu_{\theta}}$ 
is the covariance matrix of the asymptotic distribution of the vector $\frac 1{\sqrt{n}} S_n\widetilde f^\theta$ under $\mu_\theta$. The matrix  $ \Sigma^{\mu_{\theta}}$ is positive definite since the family $\{f_i:\ i=1,\ldots, d\}$ is linearly independent as cohomology classes (Lemma \ref{lem:3.5}).
Let  $G:\mathbb R^d\to\mathbb R^{d\times d}$ be given as before with $K=\Sigma^{\mu_{\theta}}$, i.\,e.\@  
$ G(M) =M^t M - \Sigma^{\mu_{\theta}} $.
Then by Lemma \ref{lem:4.3}, we have that  $G(M)$
is $\P_{\mathcal{N}(0,\Sigma^{\mu_{\theta}})}$-a.\,s.\@  invertible  and one obtains that $\sqrt{n}(\widehat \theta_n-\theta)$ is asymptotically equivalent to the random sequence
\[ G^{-1}\left(\left(\frac 1{\sqrt{n}} S_n\widetilde f_i^\theta\right)_{1\le i\le d}\right) \left(\frac 1{\sqrt{n}} S_n\widetilde f_1^\theta,\ldots ,\frac 1{\sqrt{n}} S_n\widetilde f_d^\theta\right)^t.\] 
Applying \cite[ Theorem 2.7]{Bi},  it follows that the limiting distribution of $\sqrt{n}(\widehat \theta_n-\theta)$
is given by the distribution of $ G^{-1}(N) N^t$ with   $N\sim \mathcal{N}(0,\Sigma^{\mu_{\theta}})$.
\end{proof}

\begin{lemma}\label{lem:4.5} Let $I^n(t)$ denote the Fisher information in (\ref{eq:3.11}) for the distribution $\nu_t$, $n\ge 1$. Then there exists $\Sigma^{\nu_t}= \left(  \Sigma^{\nu_t}_{ij}\right)_{1\le i,j\le d}$ such that 
\[ \lim_{n\to\infty} \frac 1n  I^n_{i,j}(t) = \Sigma_{i,j}^{\nu_t}.\] 
Also
\[ \lim_{n\to\infty} \int (\widehat \theta_n-\theta)^t (\widehat \theta_n-\theta) \d\mu_\theta= \lim_{n\to\infty} \frac 1n E[ N^t(G^{-1}(N))^2 N] \] 
where $N$ is as in Proposition \ref{prop:4.4}.
\end{lemma}
\begin{proof} Recall that by Lemma \ref{lem:3.4} for any cylinder $C$ of length $n$
\[ D_i[\log\nu_s(C)](t) = \frac 1{\nu_t(C)}\int_C S_n\widetilde f_i^t \d\nu_t+ O(1).\] 
This term is asymptotically equivalent to $S_n\widetilde f_i^t(x)$ for any $x\in C$, hence
\begin{eqnarray*}
\lim_{n\to\infty} \frac 1n I^n_{i,j}(t)&=& \lim_{n\to \infty}\frac 1n  \int S_n\widetilde f_i^t(x)S_n\widetilde f_j^t(x) \d\nu_t(x)\\
&=&\lim_{n\to \infty}\frac 1n  \int S_n\widetilde f_i^t(x)S_n\widetilde f_j^t(x) \varphi_t^{-1}\d\mu_t(x)\\
&\eqqcolon &\Sigma_{i,j}^{\nu_t}
\end{eqnarray*}
exists. This is so since $Z_n=n^{-1}   S_n\widetilde f_i^t(x)S_n\widetilde f_j^t(x)$ converges weakly to some $Z$ under $\mu_t$ by Theorem \ref{theo:3.10}, hence $(\varphi_t,Z_n)$ converges weakly to the product distribution of $\varphi_t$ and $Z$. Moreover the inner product on $\mathbb R^2$ is a continuous function. Alternatively, one may use the $\phi$-mixing properties here.

The last statement is standard. Let ${\bf S}_n=(\frac 1{\sqrt{n}} S_n\widetilde f_i^\theta)_{1\le i\le d})$
\begin{eqnarray*}
E_{\mu_\theta}[(\widehat \theta_n-\theta)^t(\widehat \theta_n-\theta)] &=&
E_{\mu_\theta}[ (G^{-1} ({\bf S}_n)) {\bf S}_n)^t G^{-1} ({\bf S}_n)) \bf S_n]\\
&=&
 E_{\mu_\theta} [{\bf S}_n^t (G^{-1} ({\bf S}_n))^2 {\bf S_n}].
\end{eqnarray*}
\end{proof}

\begin{proof}[Proof of Theorem \ref{theo:2.4}]  This follows from Lemma \ref{lem:4.1} and Proposition \ref{prop:4.2}. 
\end{proof}
 \begin{proof}[Proof of Theorem \ref{theo:2.8}] ~ 
\begin{description}
\item[Part  1.] follows from  Theorems \ref{theo:2.4}, \ref{theo:2.6}  and Proposition \ref{prop:4.2}.

\item[Part  2.] The first statement is the well-known central limit theorem for vectors of H\"older continuous functions under mixing Gibbs measures. The remaining part follows from Lemma \ref{lem:4.3}.

\item[Part  3.] follows from Proposition \ref{prop:4.4}.

\item[Part  4.] follows from the Cram\'er--Rao inequality and the continuity of the operator $G^{-1}(n^{-1/2}S_n\tilde f_i)_{1\le i\le d})$.
\end{description}
The additional assertion is in Lemma \ref{lem:4.5} as well.
\end{proof}

\section{Maximum likelihood test}\label{sec:5}

In this section we discuss the application of the maximum likelihood estimator for constructing test statistics about the parameter $\theta$.
Thus we keep the statistical model of a family of Gibbs measures $\nu_\theta$, $\theta\in \Theta\subset \mathbb R^d$ given by their potentials $F_\theta= f_0+\sum_{i=1}^d \theta_if_i$ as in Section \ref{sec:2}. For the test problem $H_0\subset \Theta$ against $H_1=\Theta\setminus H_0$ based on observations $\X_{n}\coloneqq X_0,\ldots ,X_{n-1}$, the maximum likelihood  test is defined by its region of rejection
\[\frac {L_n^{H_0}(X_0,\ldots ,X_{n-1})}{L_n(X_0,\ldots ,X_{n-1})}\le c\] 
where $c<1$ (cf.\@ (\ref{eq:2.1}) for the definition of $\Theta_{n}^{\eta_{n}}(\X_n)$),
\[ L_n^{H_0}(\X_n):=\max_{\theta\in H_0\cap\Theta_{n}^{\eta_{n}}(\X_n)}  \nu_\theta([\X_n])\; \mbox{ and }
  L_n(\X_n):=\max_{\theta\in \Theta_{n}^{\eta_{n}}(\X_n)}  \nu_\theta([\X_n]).
 \]
As explained in the introduction (Section \ref{sec:1}) this implies that the test problem is adapted to the situation when only $X_0,\ldots, X_{n-1}$ are observed.

This means that the test statistic is given by the quotient of the likelihoods as in Section \ref{sec:2} for the parameter spaces $H_0$ and $\Theta$ respectively. The constant $c$ determines the significance level, which increases with $c$.

The simplest case of the test problem is when hypothesis and alternative are simple. In this case the Neyman-Pearson test is optimal as explained and investigated in \cite{DLL}. Clearly, in this case the maximum likelihood test reduces to the Neyman-Pearson test statistic.
Thus
\begin{proposition}\label{prop:5.1}
The maximum likelihood test is a Neyman-Pearson test when restricted to  the test problem $\theta_0\in H_0$ against $\theta_1\not\in H_0$ for every distinct pair $\theta_i\in \Theta$, $i=0,1$.
\end{proposition}

\subsection{Simple hypothesis}\label{sec:5.1}
As a first application of our main result Theorem \ref{theo:2.8} we treat the special case of $H_0:=\{\theta^0\}\subset \Theta$. This means that the null hypothesis consists of a single probability $\nu_{\theta^0}$. 
 \begin{example}\label{ex:5.2} Take $\theta^0$ to be a parameter satisfying  $P(F_{\theta^0})=0$. This parameter is often unique and in some cases represents the Hausdorff dimension of an attractor. For example take $f_0=0$ and $f_1= \log |R'|$ where $R$ is a hyperbolic differential map, like a rational map on the Riemann sphere. This then follows from the Bowen-Manning-McCluskey formula (see \cite{D3} for background and more details). Hence the hypothesis $H_0$ serves to test whether an observed system  is in that state, and, moreover, one may use the MLE to estimate the Hausdorff dimension of its attractor once it is known that the system has vanishing pressure.
\end{example}
If $n$ is large enough, that is if the maximum likelihood estimators are within the region of $[-\eta_n^{-1},\eta_n^{-1}]^d$, then
\begin{eqnarray*} L_n^{H_0}(X_0,\ldots ,X_{n-1})&=& \nu_{\theta^0}([X_0,\ldots ,X_{n-1}])\\
&=& \int \e^{-nP(F_{\theta^0})+S_n(F_{\theta^0}(\X_nx))} \d\nu_{\theta^0}(x)
\end{eqnarray*}
and --using Taylor expansion as in the proof of Proposition \ref{prop:4.4}--
\begin{equation}\label{eq:5.1}
 L_n(\X_n)= \nu_{\hat\theta}(\X_n)= \nu_{\theta^0}(\X_n)
  +\sum_{|\lambda|=1}^\infty D_\lambda [\nu_s([\X_n]](\theta^0) \prod_{i=1}^d (\hat\theta_j-\theta_{j}^0)^{k_j}.
 \end{equation} 
  \begin{proposition}\label{prop:5.3}  Let $f_0,f_1,\ldots,f_d$ be H\"older continuous, $F\coloneqq(f_1,\ldots,f_d)$ be linearly independent as cohomology classes and $\X_{n}\coloneqq(X_0,\ldots ,X_{n-1})$. Then, under $\mu_{\theta^0}$, the distributions of
 \[ \frac1{\nu_{\theta^0}([\X_{n}])} \sum_{|\lambda|=1}^\infty D_\lambda [\nu_s([\X_{n}]](\theta^0) \prod_{i=1}^d (\hat\theta_j-\theta_{j}^0)^{k_j}\] 
  converge, as $n\to\infty$,  weakly to the distribution of 
 \[ Z=\sum_{i=1}^d N_iZ_i,\] 
 where $Z_i$ denotes the $i$-th coordinate of $G^{-1}(N)N$ and where $N$ is the limiting normal distribution for $\mu_{\theta^0}$ derived in Theorem \ref{theo:2.8}.
 \end{proposition}
 \begin{proof} Fix $n$ and let $\hat\theta$ denote the maximum likelihood estimator for $n$. First observe that
 \begin{eqnarray*}
 &&  \frac1{\nu_{\theta^0}([\X_{n}])} \sum_{|\lambda|=1}^\infty D_\lambda [\nu_s([\X_{n}])](\theta^0) \prod_{i=1}^d (\hat\theta_j-\theta_j^0)^{k_j} =\\
 &&\qquad
\sum_{i=1}^d \frac{D_i[\nu_s([\X_{n}])](\theta^0)}{\nu_{\theta^0}([\X_{n}])}(\hat\theta_i-\theta_i^0)+ O(\| \hat\theta-\theta^0\|^2).
\end{eqnarray*}
Under $\mu_{\theta^0}$, by Theorem \ref{theo:2.8} $\sqrt{n} (\hat\theta-\theta^0)$ converges weakly to the distribution of  $G^{-1}(N)N$ where $N$ has a normal distribution with expectation $0$ and covariance matrix $\Sigma^{\mu_{\theta^0}}=\Cov_{\mu_{\theta_{0}}}(F)$. 
In fact $N$ is the distributional limit of the random vector 
\[ \frac 1{\sqrt{m}} \left(S_m\tilde f_i^{\theta^0}\right)_{1\le i\le d},\] 
as shown at the end of the proof of Proposition \ref{prop:4.4}.
Moreover, by Lemma \ref{lem:3.3}, (\ref{eq:3.4})
\[ \frac{D_i [\nu_s([\X_{n}])](\theta^0)} {\sqrt{n} \nu_{\theta^0}([\X_{n}])}= \frac 1{\sqrt{n}}S_n\tilde f_i^{\theta^0}+ o(1)\] 
for every $1\le i\le d$.
Hence
\begin{equation*}
   \frac1{\nu_{\theta^0}([\X_{n}])} \sum_{|\lambda|=1}^\infty D_\lambda [\nu_s([\X_{n}])](\theta^0) \prod_{i=1}^d (\hat\theta_j-\theta_j^0)^{k_j} 
\end{equation*}
has the same limiting distribution as 
\[\sum_{i=1}^d \left(\frac 1{\sqrt{n}} S_n\tilde f_i^{\theta^0}\right)\left(G^{-1}\left(\left(\frac
 1{\sqrt{n}} \tilde S_nf_l^{\theta^0}\right)_{1\le l\le d}\right) \left(\frac 1{\sqrt{n}} S_n\tilde f_l^{\theta^0}\right)_{1\le l\le d}\right)_i,\] 
 which is the distribution of the random quadratic form
\[ \sum_{i=1}^d N_i(G^{-1}(N)N)_i.\] 
 \end{proof}
 \begin{remark}\label{rem:5.4} Since $\mu_{\theta^0}$ has the known potential function $F_{\theta^0}$ the matrix $\Sigma^{\mu_{\theta^0}}$ can be calculated with arbitrary precision from the transfer operator. This implies that the quantiles of the random quadratic form in the proposition can be determined by 
 \[\mathrm{Prob}\left(\sum_{i,j=1}^d N_i G^{-1}(N)_{ij} N_j\le c\right)=\alpha.
 \]
 \end{remark}
 \begin{theorem}\label{theo:5.5} For $0<\alpha<1$ and  $\X_{n}\coloneqq(X_0,\ldots ,X_{n-1})$  let 
 \[ \left\{L_n^{H_0}(\X_{n})\le c_n L_n(\X_{n})\right\}\] 
 be the rejection region for the maximum likelihood test at level $\alpha$ with hypothesis $H_0=\{\theta^0\}\subset \Theta$. Then
 \[ \lim_{n\to\infty} c_n=\frac 1{z_\alpha+1},\] 
 where $z_\alpha$ denotes the lower $\alpha$-quantile of the limiting distribution of the random quadratic form 
 \begin{equation}\label{eq:5.2}
 \Xi(N,N)=\sum_{i,j=1}^d N_i(G^{-1}(N))_{ij}N_j
 \end{equation}
  derived in Proposition \ref{prop:5.3}.
 \end{theorem}
 \begin{proof} By equation (\ref{eq:5.1})
 the inequality defining the rejection region is
 \begin{eqnarray*}
 \frac{L_n(\X_{n})}{L_n^{H_0}(\X_{n})}= \frac{\nu_{\theta^0}([\X_{n}])+ \sum_{1=1}^d D_i[\nu_s([\X_{n}])](\hat \theta_i-\theta_i^0)}{\nu_{\theta^0}([\X_{n}]}+o(1) \ge c_n^{-1}
 \end{eqnarray*}
   one deduces that $c_n^{-1}-1$ is given by 
 \[ \mu_{\theta^0}\left( \sum_{1=1}^d \frac{D_i[\nu_s([\X_{n}])](\theta^0)(\hat\theta_i-\theta_i^0)}{\nu_{\theta^0}([\X_{n}]}\ge c_n^{-1}-1\right)=\alpha.\] 
 If $z_\alpha $ denotes the lower $\alpha$-quantile of $\Xi$, i.\,e.
 \[ \mathrm{Prob}(\Xi(N,N)\ge z_\alpha)=\alpha,\] 
 then by Proposition \ref{prop:5.3}
 \[ \lim_{n\to\infty} (c_n^{-1}-1)= z_\alpha.\] 
 \end{proof}
 
 \subsection{Testing for vanishing influence}\label{sec:5.2}
  
  Similar to the derivative of the test region in the previous subsection
  we can proceed for the hypotheses 
  \[ H_0^k=\{\theta\in \Theta: \theta_k=0\},\] 
   where $1\le k\le d$ is fixed, with some modifications.  The idea is to estimate $\theta$ in the family $H_0$ by some $\hat \theta^0$ and then apply the method in Section \ref{sec:5.1} to the testing problem with the simple hypothesis $\big\{\hat\theta^0\big\}$. This leads to the following definition.
  \begin{definition}\label{def:5.6}
  A modified maximum likelihood test for $H^k_0$ ($1\le k\le d$) based on the observations $\X_{n}\coloneqq(X_0,\ldots ,X_{n-1})$ has a region of rejection of the form
  \[\left\{\frac{L_n^{H_0^k}(\X_{n})}{L_n(\X_{n})}\le c(\X_{n})\right\}.\]
  \end{definition}
  In order to obtain the asymptotic relation of the region in dependence of the significance level, one needs to estimate the parameters in the asymptotic distributional limit under the null hypothesis, that is estimation of the covariance matrix $\Sigma^{\mu_{\hat\theta^0}}$ where $\hat\theta^0$ denotes the estimated parameter in $H_0^d$. This leads to a critical value $c$ which depends on the observation. For simplicity we may assume $k=d$ and write $H_0$ for $H_0^d$.
 \begin{theorem}\label{theo:5.7}  Let $f_0,f_1,\ldots,f_d$ be H\"older continuous function,  $F\coloneqq(f_1,\ldots,f_d)$ be  linearly independent as cohomology classes and $\X_{n}\coloneqq(X_0,\ldots ,X_{n-1})$.
 The modified maximum likelihood test for $H_0=\{\theta_d=0\}$ against $H_1=\{\theta_d\ne 0\}$ at significance level $\alpha$ is defined by its region of rejection
 \[
 \left\{L_n^{H_0}(\X_{n})\le c_n(\X_{n}) L_n(\X_{n})\right\},
 \]  
 where 
 \[ c_n(\X_{n})= \frac 1{1+z_n(\X_{n})}\] 
 with $z_n=z_n(\X_{n})$ defined as the $\alpha$-quantile of the limiting distribution $ \Xi(\hat N,\hat N)$ where $\hat N$ is the centered normal distribution with covariance matrix $\hat\Sigma=\Sigma^{\mu_{\hat\theta^0}}\coloneqq \Cov_{\mu_{\theta^{0}}}(F)$. 
 Under $\mu_\theta$, $\theta\in H_0^d$ and for the significance level $\alpha\in (0,1)$ the random function $c_n$ satisfies
 \[\lim_{n\to\infty} c_n(\X_{n}) = \frac 1{z_\alpha+1},\] 
 with (cf. (\ref{eq:5.2}))
 \[ \mbox{\rm Prob}(\Xi(N,N)\le z_\alpha)=\alpha\] 
 where $N$ is the centered normal distribution with covariance  $\Sigma^{\mu_{\theta}}=\Cov_{\mu_{\theta}}(F)$.
 \end{theorem} 
 \begin{remark}\label{rem:5.8} 
 A consistent estimator for the  true limiting covariance matrix is given by  estimating the parameter $\theta\in\Theta$ or its moment estimator
\begin{eqnarray*}
&& \frac 1n \sum_{k,l=0}^{n-1} \left(f_i(X_k,\ldots ,X_{n-1},z)-\frac 1n\sum_{s=0}^{n-1} f_i(X_s,\ldots ,X_{n-1},z)\right)\\
&&\qquad \left(f_j(X_l,\ldots ,X_{n-1},z)-\frac 1n\sum_{s=0}^{n-1} f_j(X_s,\ldots ,X_{n-1},z)\right),
\end{eqnarray*}
where $z$ is any pre-assigned point in $\Omega$.
This follows from the ergodic theorem, the $\psi$-mixing property of Gibbs measures and the H\"older continuity of the functions $f_i$, $1\le i\le d$ by standard arguments.
 \end{remark}
 \begin{proof} Let $\theta\in H_0$ and $\Sigma\coloneqq \Cov_{\mu_{\theta}}(F)$ be the associated  covariance matrix.  
 By assumption, it is consistently estimated by some sequence of estimators  $ \hat \Sigma_{ij}^n$.
 
 Its type 1 error of the modified maximum-likelihood test  at $\theta\in H_0$   satisfies with $\X_{n}=(X_0,\ldots ,X_{n-1})$,
 \begin{eqnarray*}
 && \lim_{n\to\infty}\mu_{\theta}\left(\frac {L_n^{H_0}(\X_{n})}{L_n(\X_{n})}\le c_n(\X_{n}) \right) =\lim_{n\to\infty}\mu_\theta\left( \frac {\nu_\theta([\X_{n}])}{L_n(\X_{n})} \le c_n(\X_{n}) \right).
 \end{eqnarray*}
 By Theorem \ref{theo:5.5} there is $c_n$ such that
 \[ \mu_\theta\left( \frac {\nu_\theta([\X_{n}])}{L_n(\X_{n})} \le c_n \right)=\alpha\] 
 and 
 \[\lim_{n\to\infty} c_n(\X_{n})= \lim_{n\to\infty}\frac 1{z_n(\X_{n})+1}\] 
  with $z_n(\X_{n})$ being the lower $\alpha$-quantile of $\Xi(\overline N,\overline N)$ where $\overline N$ denotes the limiting normal distribution associated to the estimated parameter $\hat\theta^0\in H_0^d$.
  
  Let
  $ z_\alpha$ be defined as the lower $\alpha$-quantile of
 $\Xi(N,N)$ where $N$ is the limiting normal distribution under $\mu_\theta$.
 Since under $\mu_\theta$ we have that  $\hat \theta^0\to \theta$  and by the continuity of integrals $t\mapsto\int g \d\mu_t$ for H\"older functions $g$, we have that  the covariance matrix associated to $\hat \theta^0$ converges to the covariance matrix associated to $\theta$. Finally, since the degenerate centered normal distributions depend continuously  on their covariance matrices, we conclude that $\lim_{n\to\infty} z_n=z_\alpha$,
  proving  the theorem.
 \end{proof}
\begin{remark}\label{rem:5.9}  The same conclusion holds when $c_n(\X_{n}) $ is determined by the normal distribution associated to a consistent estimator for $\hat N$, that is when the associated covariance matrix is estimated consistently, see Remark \ref{rem:5.8}.
\end{remark}

\section{Applications}\label{sec:6}

The main theorem (Theorem \ref{theo:2.8}) may be used to derive many statistical decision procedures. Here we mention a few and work them out in detail. In addition we also connect the result to the classical setup of i.i.d. samples. 

To begin with one can construct asymptotic level $\alpha$ confidence intervals, that is random variables $U_n=U_n(\X_{n})$ and $V_n=V_n(\X_{n})$ such that for any $\theta\in \Theta$
\[  \lim_{n\to\infty}\mu_\theta(\{ L_n\le \theta\le U_n\})=\alpha.\] 
\begin{example}\label{ex:6.1}
For $n\in\mathbb N$ large enough let $\hat\theta^n$ denote the maximum likelihood estimator based on the observations $X_0,\ldots ,X_{n-1}$. Under $\mu_\theta$ they converge to the correct value $\theta$. Let $\hat N^n$ denote the centered normal distribution with estimated covariance structure $\hat \Sigma^n=\Sigma^{\mu_{\hat\theta^n}}$ derived from the observation $X_0,\ldots ,X_{n-1}$ and the maximum likelihood estimator $\hat \theta^n$ (note that this parameter is needed to center the functions $f_i$). For $1\le i\le d$ choose the upper and lower $\alpha/{2d}$ quantiles in the form $U_i^n\sqrt{n}$ and $V_i^n\sqrt{n}$ of the distribution of the $i$-th coordinate of the random vector $G^{-1}(\hat N^n)\hat N^{n}$, to be more precise, for $1\le i\le d$ let
 \[\mathrm{Prob}\left\{-\sqrt{n}U_i^n\le (G^{-1}(\hat N^n)\hat N^n)_i\le \sqrt{n}V_i^n\right\} =\alpha/d\qquad \forall 1\le i\le d.
 \] 
Then by Theorem \ref{theo:2.8}
\begin{eqnarray*}
&& \lim_{n\to\infty}\mu_{\theta}\left(\bigcap_{i=1}^d \left\{ \hat\theta_i^n-V_i^n\le \theta_i\le \hat\theta_i^n+U_i^n\right\}\right)\\
 &&=\lim_{n\to\infty}\mu_\theta \left(\bigcap_{i=1}^d\left\{ -\sqrt{n}U_i^n\le \sqrt{n}(\hat \theta_i^n-\theta_i)\le  \sqrt{n}V_i^n\right\}\right)\\
 && \le \sum_{i=1}^d \lim_{n\to\infty}\mu_\theta \left(\left\{ -\sqrt{n}U_i^n\le \sqrt{n}(\hat \theta_i^n-\theta_i)\le  \sqrt{n}V_i^n\right\}\right)\to\alpha.
 \end{eqnarray*}
 Here the confidence region is a rectangle in the $L_1$-norm where the size decreases like $1/\sqrt{n}$ as $n\to\infty$. Similarly of course, one can obtain confidence regions in the  euclidean  $L_2$-norm.
\end{example}
\begin{example}\label{ex:6.2}
 We explain Theorem \ref{theo:2.8} in case $d=1$, i.\,e.\@ $F_t=f_0+tf$. 
 The function $G$ becomes
 \[G(m)= m^2- \sigma^2\qquad m\in \mathbb R\] 
 where - abusing the notation $\sigma^2$ here - $\sigma^2=\lim_{n\to\infty} \frac 1n \int (S_n\widetilde f^t)^2 \d\mu_t$ denotes the asymptotic variance of the random sequence $\frac 1{\sqrt{n}} S_n(f-\int f \d\mu_t)$. Therefore, 
 \[ G^{-1}(m)= \frac 1{m^2-\sigma^2}\] 
 and
 \[ \sqrt{n}(\widehat \theta_n -\theta) = \frac {\sqrt{n}}{ (S_n\widetilde f^\theta)^2 - n \sigma^2} S_n\widetilde f^\theta+ o(1).\] 
 It follows that the limit distribution is given by a random variable
 \[ \frac {N}{N^2-\sigma^2},\] 
 where $N$ is a centered normal distribution with variance $\sigma^2$.
 
 The Fisher information is just $n\sigma^2$ so that
 \[ \int (\widehat \theta_n-\theta)^2 d\mu_\theta=  \int \frac 1{ ((S_n\widetilde f^\theta)^2 - n \sigma^2)^2} (S_n\widetilde f^\theta)^2 \d\mu_\theta +o(1).\]  
 Note that this example shows the relation between the Fisher information with respect to $\mu_\theta$ and the asymptotic variance of the maximum likelihood estimator when maximizing the density of $\nu_\theta$ instead of $\mu_\theta$.
 \end{example}
 
 \begin{example}\label{ex:6.3}
  We now turn to the case $d=2$, so $F_\theta= f_0+\theta_1 f_1+\theta_2 f_2$.  We assume the standing condition that $F=(f_1,f_2)$ are linearly independent as cohomology classes. By Proposition \ref{prop:4.4} we have 
 \[  \frac 1{\sqrt{n}}(S_n\tilde f_1^\theta, S_n\tilde f_2^\theta))\to N=(N_1,N_2)\] 
 where $N$ is a  normal random variable with expectation $0$.
 Let $\Sigma^n$ denote the covariance matrix  of $(S_nf_i)_{ i=1,2}$. Then the matrix 
 \[G(m)=\left(\begin{matrix} m_1^2 & m_1m_2\\ m_1m_2 & m_2^2\end{matrix}\right)-\Sigma^n\qquad m=(m_1,m_2)\in \mathbb R^2\] 
has the inverse $\left(\begin{matrix} a& b\\c& d\end{matrix}\right)$ which satisfies
 \begin{eqnarray*}
 && a(m_1^2-\Sigma_{11}^n)+c (m_2m_1 -\Sigma_{12}^n)=1;\quad  b(m_1^2-\Sigma_{11}^n)+d(m_2m_1-\Sigma_{12}^n)=0;  \\
 &&a(m_1m_2-\Sigma_{21}^n)+c(m_2^2-\Sigma_{22}^n)=0; \quad b(m_1m_2-\Sigma_{21}^n)+d(m_2^2-\Sigma_{22}^n)=1,
 \end{eqnarray*}
 hence, using $\Sigma_{12}^n=\Sigma_{21}^n$,
 \begin{eqnarray*}
     &&  a=\frac {m_2^2-\Sigma_{22}^n}{(m_1^2-\Sigma_{11}^n)(m_2^2-\Sigma_{22}^n)- (m_2m_1 -\Sigma_{12}^n)^2 }\\
    &&  b=\frac {m_1m_2-\Sigma_{12}^n}{(m_1m_2-\Sigma_{21}^n)^2- (m_1^2 -\Sigma_{11}^n)(m_2^2-\Sigma_{22}^n) }\\
 && c=-\frac {(m_2^2-\Sigma_{22}^n)(m_1m_2-\sigma_{21}^n)}{(m_1^2-\Sigma_{11}^n)(m_2^2-\Sigma_{22}^n)^2- (m_2m_1 -\Sigma_{12}^n)^2 }\\
 && d= \frac {m_1^2-\Sigma_{11}^n}{(m_1m_2-\Sigma_{21}^n)^2- (m_1^2 -\Sigma_{11}^n)(m_2^2-\Sigma_{22}^n) }
 \end{eqnarray*} 
 Note that  $m_i:=\frac 1{\sqrt{n}} S_n\tilde f_i^\theta$ is asymptotically normal and that
 $  m_im_j $ is a von Mises statistic with kernel $k(x,y)=f_i(x)f_j(y)$. It is not centered, but the kernel $k(x,y)-\int \int f_i(u)f_j(v) \mu_\theta(du)\mu_\theta(dv)$ satisfies the assumption of Theorem \ref{theo:3.10} 
  so that we conclude that
 \[ m_im_j-\Sigma_{ij}^n\] 
 converges to a weighted $\chi^2$-distribution. This remark shows that the distribution of $G^{-1}(N)N^t$ can be explicitly simulated, and thus all test and estimate results become practicable.
 \end{example} 

\begin{example}\label{ex:6.4} Consider the family of Bernoulli measure on the full shift space $\{1,\ldots ,a\}^\mathbb N$.  Then 
for each $n\ge 1$ the family can be represented as a parametrized family of measures on $\{1,\ldots ,d\}^n$ by
\begin{equation}\label{eq:6.1}
\mu_\pi(X_0=c_0,\ldots ,X_{n-1}=c_{n-1})= \prod_{i=0}^{n-1} \pi_{c_i}
\end{equation}
so $\Pi=\{\pi=(\pi_i)_{1\le i\le d}: \theta_i\ge 0,\ \sum_i\pi_i=1\}$is the parameter space.
The densities are  given by (\ref{eq:6.1}).
Below we estimate the probabilities $P(X_0=i)=\pi_i$ using these densities and compare the result with that in the present note.

Before doing this, we recall that $\mu_{\pi}$ are equilibrium measures for the family $\sum_{i=1}^d (\log \pi_i) \1 _{[i]}$, so, denoting $g_i= \1 _{[i]}$, the potential becomes $G_\pi=\sum_{i=1}^d (\log \pi_i) g_i$. This fact is well known, but we repeat the argument here, claiming that $\mu_\pi$ is the eigenmeasure of the transfer operator with potential $G_\pi$. Let $T$ denote the shift transformation on $\{1,\ldots ,d\}^{\mathbb N_0}$ and $h_\mu(T) $ the metric entropy of the invariant measure $\mu$. First observe that
\[ h_{\mu_\pi}(T)+\int G_\pi \d\mu_\pi= -\sum_{i=1}^d \pi_i\log \pi_i+\sum_{i=1}^d (\log \pi_i) \mu_\pi([i])= 0.\] 
For general invariant and ergodic measures $\mu$ we have that
\[ h_\mu(T)\le H_\mu(\{[i]: 1\le i\le d\})= -\sum_{i=1}^d \mu([i])\log \mu([i])\] 
with equality if and only if $\mu$ is a Bernoulli measure. Therefore - if $\mu$ is not a Bernoulli measure
\begin{eqnarray*}
&& h_\mu(T)+\int \log F_\pi \d\mu < -\sum_{i=1}^d \mu([i])\log \mu([i]) + \sum_{i=1}^d \mu([i])\log \pi_i \\
&& = \sum_{i=1}^d \mu([i]) \log \frac{\pi_i}{\mu([i])}\le \log \sum_{i=1}^d \mu([i]) \frac{\pi_i}{\mu([i])} =0
\end{eqnarray*}
If $ \mu$ is a Bernoulli measure and not $\mu_\pi$ the same estimation works, replacing the first $<$ sign by $\le$ and the last $\le$ sign  by $<$. Thus $\mu_\pi$ is the equilibrium measure. Since $1=\mu_\pi(T[i])=\int \sum_{j=1}^d \1 _{[j]} \e^{G_\pi}\d\mu_{\pi}=\sum_{j=1}^d \pi_j$,  $\mu_\pi$ has the potential $G_\pi$.

Note that the functions $(g_i:1\le i\le d)$ are not linearly independent, hence also not linearly independent as cohomology classes. 
In fact $\1 _{[d]}=1-\sum_{i=1}^{d-1} \1 _{[i]}$ and $\pi_d=1-\sum_{i=1}^{d-1} \pi_i$ so that the functions
\[ f_i= \1 _{[i]}\qquad 1\le i <d\]  
satisfy
\[F_\pi=\sum_{i=1}^{d-1} \log\frac{\pi_i}{1-\sum_{j=1}^{d-1}\pi_j} \1 _{[i]} = G_\pi  -\log \left(1-\sum_{j=1}^{d-1}\pi_j\right)\] 
and $F_\pi$ is therefore a potential cohomologuous  to $G_\pi$.
Since the functions $f_i$ are linearly independent, we arrive at the model in Section \ref{sec:2} with the new parametrization
\[ \Theta=\left\{\theta=\Big(\log ({\pi_i}/(1-\sum_{j=1}^{d-1}\pi_j) )\Big)_{1\le i < d}:\pi\in \Pi \right\}.
\]
We have shown that we can formulate the estimation problem in the framework of Section \ref{sec:2}.  It follows from Theorem \ref{theo:2.8} that the maximum likelihood estimators $\widehat\theta^n$ exists for the parameter  $\theta$ 
and --under $\mu_\theta=\mu_\pi$-- has the asymptotic distribution $G^{-1}(N)N^t$ where $N$ is a centered normal distribution with  covariance matrix   $\Sigma=\Sigma^{\mu_\theta}$.

Note that we can obtain an estimator for $\pi$ by solving the system of equations
\begin{equation}\label{eq:6.2}
  \theta_i =\log \frac{\pi_i}{1-\sum_{j=1}^{d-1}\pi_j}\quad 1\le i <d,
  \end{equation}
and   $ \pi_d= 1-\sum_{j=1}^{d-1} \pi_j$.

The maximum likelihood estimator for $\pi$ is well known when using the densities (\ref{eq:6.1}).

In order to compare the two methods, let $d=2$, that is we estimate whether a coin is fair or not. In this case (\ref{eq:6.2}) entails (omitting the sample size $n$ in the notation)
\[ 
\widehat\pi_1=\frac{\e^{\widehat\theta_1}}{1+\e^{\widehat\theta_1}}\] 
and, by  Example \ref{ex:6.2}, for the true parameter $\theta$
\[\sqrt{n}(\widehat\theta_1-\theta_1)\to_{n\to\infty} \frac N{N^2-\sigma^2}\] 
where $N$ is a centered normal distribution with variance 
\begin{eqnarray}\label{eq:6.3}
\sigma^2&=&\lim_{m\to\infty} \frac 1m \int (S_m\tilde f_1^\theta)^2 \d\mu_\theta= \pi_1(1-\pi_1).
\end{eqnarray}
On the other hand, differentiating equation (\ref{eq:6.1}) with respect to  $\pi_1$ yields
\begin{eqnarray*}
&& \frac \partial{\partial \pi_1}\mu_\pi([c_0,\ldots ,c_{n-1}])= \frac{d}{d\pi_1}\pi_1^{S_ng_1}(1-\pi_1)^{n-S_ng_1}\\
 &=& \left(\frac{S_ng_1}{\pi_1}-\frac{n-S_ng_1}{1-\pi_1}\right)\exp\left(S_ng_1\log \pi_1 + (1-S_ng_1)\log (1-\pi_1)\right)
 \end{eqnarray*}
which vanishes for 
\[ \pi_1= \frac 1n S_ng_1.\] 
Hence the maximum likelihood estimator $\widehat\pi$ for $\pi$,
\[ \sqrt{n}(\widehat\pi_1-\pi_1),\] 
is asymptotically normal with expectation $0$ and variance
$\tau^2= \pi_1(1-\pi_1)$. 
Since by (\ref{eq:6.2}) $\theta_1=  \log\frac{\pi_1}{1-\pi_1}$ we obtain for the estimator $\overline\theta_1=  \log\frac{\widehat\pi_1}{1-\widehat\pi_1}$
\begin{eqnarray*}
&& \sqrt{n}(\overline\theta_1-\theta_1)= \sqrt{n}\left( \log\frac{\widehat\pi_1}{1-\widehat\pi_1}-\log\frac{\pi_1}{1-\pi_1}\right)
\end{eqnarray*}
Using Taylor expansion around ${\theta_1}/{(1-\theta_1)}$ up to the first order one obtains
\begin{eqnarray*}
&& \sqrt{n}(\overline\theta_1-\theta_1)=\\
 && =\sqrt{n}\left(\frac {1-\pi_1}{ \pi_1}\right)\left(\frac{\widehat\pi_1-\pi_1}{(1-\widehat\pi_1)(1-\pi_1)}\right) 
 +o(1)\\
 &&= \sqrt{n} \frac{\widehat\pi_1-\pi_1}{(1-\widehat\pi_1)\pi_1}+o(1).
\end{eqnarray*}
We conclude that the estimator for $\theta$ derived from the maximum likelihood estimator for $\pi$ differs only by the denominator: In the first case it is random and asymptotic to $\tau^2$ in the latter case it is the weighted sum of two $\chi^2$-distributions minus $\tau^2$ (cf. Theorem \ref{theo:3.10}).
In the setup of this note the family of densities of the invariant Gibbs measure $\mu_\theta$ is
\[  \exp\{P(F_\theta)-\theta S_nf_1\}\] 
\end{example}

\begin{example}\label{ex:6.5}  Consider the family of invariant Markov chains on the subshift of finite type $\Omega$. Let $A=(a_{ij})_{1\le i,j\le a}$ denote its defining $\{0,1\}$-valued transition matrix. Then the family can be described  by $
\{ p=(p_{ij})_{(i,j)\in I}\}$ where $I=\{(i,j)|\  a_{ij}=1\}$. The initial state for the chain is then given by the unique normed eigenvector $\pi=(p_i)_{1\le i\le a}$ for the eigenvalue $1$ of the matrix $P=(\tilde p_{ij})$ where $\tilde p_{ij}=p_{ij}$ for $(i,j)\in I$ and $\tilde p_{ij}=0$ otherwise.

This Markov measure $\mu_p$ is then given by
 \begin{eqnarray*}
 \mu_p([c_0,\ldots ,c_{n-1}])&=&p_{c_0}p_{c_0c_1}\ldots p_{c_{n-2}c_{n-1}}\\
 &=& p_{c_0} \exp\left(\sum_{l=0}^{n-2} \sum_{(i,j)\in I} \log(p_{ij}) \1 _{[i,j]}(\sigma^l(x))\right),
 \end{eqnarray*}
 where $x\in [c_0,\ldots ,c_{n-1}]\subset \Omega$   and where $\sigma$ denotes the left shift transformation as before.
 
 The functions $\1 _{[i,j]}$ ($(i,j)\in I$) are certainly linearly  dependent
 since their sum equals $1$. Deleting one of these functions, say $\1 _{[a,a]}$ w.l.o.g., from this family the new collection  is 
 linearly independent and any nontrivial linear combination is not cohomologous to $0$:
 If $0\ne s\in \mathbb R^{|I|-1}$, $f_1,\ldots ,f_{|I|-1} $ an enumeration of all function in the new collection and
 \[ \sum_i s_i f_i   = c + g-g\circ \sigma\] 
 then also for each $n\ge 1$
 \[ \sum_i s_i \sum_{k=0}^{n-1}f_i\circ \sigma^k= nc+ g-g\circ \sigma^n.\] 
 Dividing by $n$ it follows that
 \[ \lim_{n\to\infty}  \sum_i \frac{s_i}{n} \sum_{k=0}^{n-1}f_i\circ \sigma^k= c\]
 at every point in $\Omega$. On the other hand, by the ergodic theorem,
  \[ \lim_{n\to\infty}  \sum_i \frac{s_i}{n} \sum_{k=0}^{n-1}f_i\circ \sigma^k= \sum_i s_i \int f_i \d\mu \qquad \text{\rm $\mu$ a.s.}\]
  for every ergodic shift-invariant probability $\mu$. Since it is easy to see  that the right hand side takes different values for some invariant measures $\mu$ and $\mu'$ we get a contradiction.
  
The constraint is $\sum_{(i,j)\in I} p_{ij}=a$   which gives now
the parametrization 
\begin{eqnarray*}
&& \mu_t([c_0,\ldots ,c_{n-1}])=\\
&&p_{c_0} \exp\left(\sum_{l=0}^{n-3}\left(t_{aa}\1 (\sigma^l(x))+ \sum_{(a,a)\ne (i,j)} t_{ij} \1 _{[i,j]}(\sigma^l(x))\right)\right)
\end{eqnarray*}
where
\begin{equation}\label{eq:6.4}
 t_{ij}= \log p_{ij} -\log (a-\sum_{(l,k)\ne(a,a)} p_{lk})\quad (i,j)\ne (a,a),
 \end{equation}
\begin{equation}\label{eq:6.5}
  t_{aa}= \log (a-\sum_{(l,k)\ne(a,a)} p_{lk})
  \end{equation}
and where  we used that $\1 _{[a,a]}=\1 -\sum_{(a,a)\ne (i,j)\in I} \1 _{[i,j]}$.

The entropy $h_\mu(T)$ of a Markov measure with initial probabilities $p_i$ and transition probabilities $p_{ij}$ is well known,
\[h_\mu(T)= \sum_{(i,j)\in I} p_ip_{ij}\log p_{ij}.\] 
Moreover,
\[ \int \1 _{[i,j]} \d\mu= p_i p_{ij}\quad \text{\rm and}\quad \int \1  \d\mu=1.\] 

Using the functions $f_i$ ($1\le i\le d:=|I|-1$) as before and letting $t=(t_i)_{1\le 
i\le d}$ denote the corresponding $t_{ij}$ as above, then  for
$ F_t= \sum_{i=1}^{d} t_i f_i + t_{aa}$
we obtain that 
\begin{eqnarray*}
\int F_t \d\mu&=&  \sum_{(a,a)\ne (i,j)\in I} t_{ij} p_ip_{ij} + t_{aa}\\
&=& \sum_{(a,a)\ne (i,j)\in I} (\log p_{ij} -\log (a-\sum_{(a,a)\ne (l,k)} p_{lk})) p_i p_{ij} + t_{aa}\\
&=&  \sum_{(a,a)\ne (i,j) \in I} p_ip_{ij}\log p_{ij} -\log p_{aa}\sum_{(a,a)\ne (i,j)\in I}  p_ip_{ij} + \log p_{aa}\\
&=& \sum_{(i,j)\in I} p_ip_{ij}\log p_{ij}
\end{eqnarray*}
This means that the Gibbs measure for the potential $F_t$ is the Markov measure $\mu_p$ associated to $\pi$ and $P$,   which is  shown similarly to the arguments in Example \ref{ex:6.4}.

As in the case of Bernoulli measures the contribution of $t_{aa}$ can be estimated together with the other parameters $t_{ij}$
so that  $F_t= \1+\sum_{i=1}^d \theta_i f_i$ is a representation of the potential functions where $\theta=(\theta_i)_{1\le i\le d}$ are the newly defined parameters.

By Theorem \ref{theo:2.4} we can estimate the parameters $\theta_{i}$  ($1\le i\le d$) using the maximum likelihood estimator $\widehat\theta^n$ so that $\widehat \theta^n\to \theta$ a.\,s.\@ and the distribution  of $\sqrt{n}(\widehat \theta^n- \theta)$  under $\mu_\theta$ converges weakly to the distribution of  $G^{-1}(N) N^t$, where $N$ is a normal variable as explained in Theorem \ref{theo:2.8}. The parameters $t_{ij}$ are then estimated solving the equations defining the parameters $\theta$.
Details are left to the reader.

Solving the equations (\ref{eq:6.4}) and (\ref{eq:6.5}) for $p_{ij}$ yields
\[ p_{aa} = \exp t_{aa}\] 
and 
\[  p_{ij}= p_{aa} \exp t_{ij}=\exp (t_{ij}+t_{aa}).\] 
Hence we arrive at estimators 
\[ \widehat p_{aa}^n= \exp \widehat t_{ij}^n\quad 
\widehat p_{ij}^n=\exp(\widehat t_{ij}^n+\widehat  t_{aa}^n).
\] 
Hence the estimators for $p_{ij}$ are a  continuous function of $\widehat \theta^n$ and the asymptotic distribution carries over. Details are left to the reader (compare the discussion in Example \ref{ex:6.4})
\end{example}

\section{Modifications of Theorem \ref{theo:2.8}}\label{sec:7}

In a final section we discuss further applications of the main results. The first extension concerns the case when the potential is linearly approximated, the second one deals with maximum potential estimators when the $f_i$ are linearly independent as cohomology classes (cf. Section \ref{sec:2.2}).

\subsection{Potentials $(F_\theta)_{\theta\in \Theta}$ of class $C^3$}\label{sec:7.1}

So let $\{F_\theta: \theta\in \Theta\subset \mathbb R^d\}$ denote a parametrized family of potential functions $F_\theta:\Omega\to \mathbb R$. Equation (\ref{eq:3.2}) gives the densities to be maximized. In the general case it seems to be a major problem to show existence and uniqueness of maximum likelihood estimators, since its derivative equating to zero cannot be solved analytically nor the existence is guaranteed at all. Note that solving the basic equation analytically also was not possible in the special case of Theorem \ref{theo:2.7}. However, the existence  problem was solved using the uniqueness of Gibbs measures in topologically mixing subshifts of finite type.

For this reason we need to make the assumption that there exists a sequence of maximum likelihood estimators $\widehat \theta^n$ which under $\mu_\theta$ converge almost surely to the true parameter $\theta$. Here maximum likelihood estimator is understood to imply that
\[ D_i[\nu_s([X_0,\ldots ,X_{n-1}])](\widehat\theta^n)=0\] 
for every $1\le i\le d$, where $\nu_s$ denotes the normalized eigenmeasure of the transfer operator with potential $F_s$.
Then we have the following result which follows essentially from the proof of Theorem \ref{theo:2.8}.
\begin{theorem}\label{theo:7.1} Let $\{\mu_t: t\in \Theta\}$ be a family of equilibrium states on a subshift of finite type with potential functions $F_t$ ($t\in \Theta$). Assume that
\begin{enumerate}
\item There is a sequence of maximum likelihood estimators $\widehat\theta^n$ ($n\ge 1$) for $\theta\in \Theta$ such that 
\[ \lim_{n\to \infty} \widehat\theta^n=\theta \qquad \mu_\theta\ \mbox{ a.s.}\] 
\item The map $t\mapsto F_t$ is $C^3$.
\item The first oder partial derivatives 
\[ D_i[F_t](\theta)\qquad 1\le i\le d\] 
are linearly independent   as cohomology classes.
\end{enumerate}
Then --with respect to the distribution $\mu_\theta$--
\[ \sqrt{n} (\widehat\theta^n-\theta)\to G^{-1}(N)N^t\] 
where $N$ is a centered normal distribution on $\mathbb R^d$ with invertible covariance matrix $  \Sigma^{\mu_\theta}=\Cov_{\mu_{\theta}}(D_1[F_t](\theta),\ldots ,D_d[F_t](\theta))$ 
 and $G(M)=MM^{t}-\Sigma^{\mu_\theta}$, $M\in \mathbb R^d$.
\end{theorem} 

\begin{proof}
The maximum likelihood   estimators $\widehat\theta^n$ for $\theta$ are 
  implicitly defined by
\[
  \nu_{\widehat \theta^n}([X_0,\ldots ,X_{n-1}])=\max_{t\in \Theta} \nu_t([X_0,\ldots ,X_{n-1}]).\] 
Using equations (\ref{eq:3.2}) and (\ref{eq:3.3}), writing $C=[X_0,\ldots ,X_{n-1}]$ and taking the first derivative this is equivalent to 
\begin{equation}
\int_C (D_i[F_s](\widehat \theta^n)-nD_i[P(F_t])](\widehat\theta^n)) d\nu_{\widehat \theta^n} + O(\nu_{\widehat\theta^n}(C))=0\qquad 1\le i\le d.\label{eq:7.1}
\end{equation}
\end{proof}
\begin{lemma}\label{lem:7.2} We have
\[D_i[P(F_t)](\widehat\theta^n)= \int D_i[F_t])(\widehat\theta^n) \d\mu_{\widehat\theta^n}.\]
\end{lemma}
\begin{proof}  The map $t\mapsto F_t$ has an expansion
\[ F_t = F_{\theta}+\sum_{i=1}^d D_i[F_s](\theta)(t_i-\theta_i) + R(\tilde\theta)\|t-\theta\|^2\] 
for some $\theta$ with $\|\tilde\theta-\theta\| \le \| t-\theta\|$.
Observe that the map $f\mapsto P(f)$ is Lipschitz which implies that
\begin{eqnarray*}
&&  \frac {P(F_t)-P(F_\theta+\sum_{i=1}^d D_i[F_s](\theta)(t_i-\theta_i))}{\|t-\theta\|}\\
&&  \le K \frac{\|F_t- F_\theta-\sum_{i=1}^d D_i[P(F_s)](\theta)(t_i-\theta_i)\|}{\|t-\theta\|}=O(\|t-\theta\|),
\end{eqnarray*}
hence 
\[D_i[P(F_t)](\theta) = D_i[P(F_\theta)+\sum_{i=1}^d D_i[F_s](\theta)(t_i-\theta_i))].
\] 
 The lemma follows now from Lemma \ref{lem:3.2},
 \[ D_i[F_\theta+\sum_{j=1}^d D_j[F_s](\theta)(t_j-\theta_j)]= 
 \int D_i[F_s](\theta) \d\mu_\theta\] 
 setting $\theta=\widehat\theta^n$.
As in   Lemma \ref{lem:3.3}, one obtains
\[ D_i[\nu_t(C)](\theta)= \int_C -n\int D_i[F_t](\theta) \d\mu_\theta+S_nD_i[F_s](\theta) \d\nu_{\theta}\] 
hence also
\begin{eqnarray*}
&& D_i[\nu_t(C)](\widehat\theta^n)= D_i[\nu_t(C)](\theta)\\
&&\qquad  +\sum_{k=1}^d D_{k}[D_i[\nu_t(C)](\theta)](\widehat\theta_k^n-\theta_k)+o(n\nu_\theta(C)).
\end{eqnarray*}
The proof is completed as the proof of Theorem \ref{theo:2.8}, where the covariance matrix is given  by $\Sigma^{\mu_\theta}$.
\end{proof}

\subsection{Asymptotic distribution for the maximum potential estimator}\label{sec:7.2}

In this section we return to the MPE estimator defined in Section \ref{sec:2.2}. Recall that it is defined by 
\[ \widetilde\theta_n\equiv \mbox{\rm arg}\max\{ \{\langle\alpha_n,(1,t)\rangle-P(F_t)\}.\] 
We always assume that the functions $f_1,\ldots,f_d$ are linearly independent as cohomology classes.

\begin{theorem}\label{theo:7.3} Let $(X_n)_{n\ge 0}$ be a sequence of random variables in $L_2(\mu_{\theta_0})$, and assume that $\widetilde \theta_n$ converges a.\,s.\@ to $\theta_0$.
Then
 \[\sqrt{n}(\widetilde \theta_n-\theta_0)\] 
   converges in distribution  to a random variable $G^{-1}(N)N^t$, where where $N\sim \mathcal{N}(0,\Sigma)$ and $G^{-1}(N)$ denotes the inverse of the matrix $G(N)$ in Theorem \ref{theo:2.8}.
   Moreover, $\widetilde \theta_n$ and the MLE $\widehat\theta_n$ ($n\ge 1$) are asymptotically equivalent in the sense that their a.\,s.\@ limit and their asymptotic distributions agree.
\end{theorem}

\begin{proof} The additional statement follows from Theorem \ref{theo:2.8}, Theorem \ref{theo:2.7} and the first part.

In order to show the first part use Taylor expansion for $\nabla P(F_s)$: 
\begin{eqnarray*}
&& D_i[P(F_s)](\widetilde \theta_n)\\
&& = D_i[P(F_s)](\theta_0)+\sum_{j=1}^d D_j[D_i[P(F_s)(t)](\theta_0)((\widetilde\theta_n)_j-(\theta_0)_j) \\
&&\qquad + \sum_{|\lambda|\ge 2} D_\lambda[D_iP(F_s)(t)](\theta_0)\prod_j ((\widetilde\theta_n)_j-(\theta_0)_j)^{k_j}\\
&&= D_i[P(F_s)](\theta_0)+\sum_{j=1}^d D_j[D_i[P(F_s)(t)](\theta_0)((\widetilde\theta_n)_j-(\theta_0)_j)\\
&& \qquad +O(\|\widetilde\theta_n-\theta_0\|^2)
\end{eqnarray*}
and since $ D_i[P(F_s)](\theta_0)= \int f_i \d\mu_{\theta_0}$
\begin{eqnarray*}
  0&=& S_{\mathbb X_n}f_i- D_i[P(F_s)](\widetilde \theta_n) \\
  &=& S_{\mathbb X_n}f_i -\int f_i \d\mu_{\theta_0} - \sum_{j=1}^d D_j[D_i[P(F_s)(t)](\theta_0)((\widetilde\theta_n)_j-(\theta_0)_j) \\
  && \qquad + O(\|\widetilde\theta_n-\theta_0\|^2).
  \end{eqnarray*}
As in the proof of Theorem \ref{theo:2.8} it then follows that the random vectors
\[ \left(S_{\mathbb X_n}f_i/n -\int f_i \d\mu_{\theta_0}\right)_{1\le i\le d}\] 
and 
\[ G(\widetilde \theta_n-\theta_0)\] 
are asymptotically equivalent (i.\,e.\@ their difference converges to $0$ in probability). Since the convergence rate is $o(\frac{1}{\sqrt{n}})$ and since the matrix $G$ is a.\,s.\@ invertible by Lemma 
\ref{lem:4.3} also $\sqrt{n}(\widetilde \theta_n-\theta_0)$ and 
\[  \sqrt{n} G^{-1}\left(S_{\mathbb X_n}f_i/n -\int f_i \d\mu_{\theta_0}\right)_{1\le i\le d}\] 
have the same limiting distribution.
The second vector clearly converges in distribution to the required law. This finishes the proof.
\end{proof}


\end{document}